\documentclass[a4paper,11.5pt]{amsart}

\usepackage[headings]{fullpage}

\usepackage[backref=page]{hyperref}

\usepackage{amsfonts,graphics,amsmath,mathrsfs,amsthm,amscd,amssymb,latexsym,euscript,enumerate}
\usepackage{epsfig}
\usepackage{flafter}
\usepackage[all,cmtip,line]{xy}
\usepackage{array}
\usepackage[english]{babel}
\usepackage{overpic}
\usepackage{subfig}
\usepackage{multirow}
\usepackage{microtype}
\usepackage{wrapfig}
\usepackage{longtable}
\usepackage{supertabular}
\usepackage[shortlabels]{enumitem}
\usepackage{tikz}
\usetikzlibrary{positioning}
\usepackage{tikz-cd}
\usepackage{float}
\allowdisplaybreaks

\usepackage[dvipsnames,svgnames,table]{xcolor}
\usepackage{graphicx}
\usepackage{hyperref}
\hypersetup{
    colorlinks=true,
    linkcolor=blue,
    citecolor=blue,
    filecolor=blue,
    urlcolor=blue
}

\newcolumntype{M}[1]{>{\centering\arraybackslash}m{#1}}

\newtheorem{theorem}{Theorem}[section]
\newtheorem{lemma}[theorem]{Lemma}
\newtheorem{proposition}[theorem]{Proposition}

\newtheorem*{theorem*}{Theorem}

\theoremstyle{plain}
\newtheorem{corollary}[theorem]{Corollary}

\theoremstyle{definition}
\newtheorem{definition}[theorem]{Definition}
\newtheorem{definition-lemma}[theorem]{Definition-Lemma}

\newtheorem{remark}[theorem]{Remark}

\numberwithin{equation}{section}

\newcommand{\Gr}{\operatorname{Gr}}
\newcommand{\Hilb}{\operatorname{Hilb}}

\newcommand{\Hom}{\operatorname{Hom}}
\newcommand{\RHom}{\operatorname{RHom}}
\newcommand{\Ext}{\operatorname{Ext}}
\newcommand{\ch}{\operatorname{ch}}

\newcommand{\C}{\mathbb{C}}
\newcommand{\Z}{\mathbb{Z}}
\newcommand{\Q}{\mathbb{Q}}
\newcommand{\OO}{\mathcal{O}}
\newcommand{\mc}{\mathcal}

\newcommand{\Db}{\mathrm{D}^{\mathrm b}}

\def\P{\mathbb{P}}

\newcommand{\ses}[3]{0\lr{#1}\lr{#2}\lr{#3}\lr 0}

\newcommand{\PP}{\mathbb{P}}

\newcommand{\CC}{\mathbb{C}}

\newcommand{\ZZ}{\mathbb{Z}}

\newcommand{\cO}{\mathcal{O} }
\newcommand{\cA}{\mathcal{A} }
\newcommand{\cB}{\mathcal{B} }
\newcommand{\cC}{\mathcal{C} }
\newcommand{\cE}{\mathcal{E} }
\newcommand{\cF}{\mathcal{F} }

\newcommand{\cI}{\mathcal{I} }

\newcommand{\cL}{\mathcal{L} }

\newcommand{\cW}{\mathcal{W} }

\newcommand{\rH}{\mathrm{H} }

\newcommand\bH{\mathbf{H}}
\newcommand\bK{\mathbf{K}}
\newcommand\bM{\mathbf{M}}

\newcommand\lr{\rightarrow}

\DeclareMathOperator{\im}{Im}

\DeclareMathOperator{\rk}{rk}

\DeclareMathOperator{\codim}{codim}

\def\Pic{\operatorname{Pic}}
\def\Proj{\operatorname{Proj}}

\def\length{\operatorname{length}}

\usepackage{mathtools}

\DeclarePairedDelimiterX{\norm}[1]{\lVert}{\rVert}{#1}

\title[Low-degree rational curves on a Fano threefold]{Hilbert schemes of low-degree rational curves on a prime Fano threefold of degree $22$}

\author[K. Chung]{Kiryong Chung}
\author[D.-W. Lee]{Dae-Won Lee}
\address[Kiryong Chung]{Department of Mathematics Education, Kyungpook National University, 80 Daehakro, Bukgu, Daegu 41566, Republic of Korea}
\email{krchung@knu.ac.kr}
\address[Dae-Won Lee]{School of Mathematics, Korea Institute for Advanced Study, 85 Hoegiro, Dongdaemun-gu, Seoul 02455, Republic of Korea}
\email{daewonlee@kias.re.kr}

\date{\today}

\begin{document}

\begin{abstract}
In this paper, we give a complete description of the Hilbert schemes of rational curves up to degree $6$ on a prime Fano threefold $X$ of degree $22$. One of the key ingredients in the geometry of these Hilbert schemes is the geometry of bisecant conics associated with rational curves on $X$. As applications, we describe additional geometric features of these moduli spaces and derive Donaldson--Thomas (DT) type invariants.
\end{abstract}

\maketitle

\section{Introduction}
Let \(X\) be a general prime Fano threefold of degree $22$ with \(\Pic(X)=\Z H\). Mukai realized $X$ as a tri-isotropic Grassmannian associated with a general triple of skew-symmetric forms on a seven-dimensional vector space \cite[Theorem 3]{Muk92}. Also, the variety $X$ was described as a parameter space of a family of twisted cubic curves in $\P(A^\vee)$, $\dim A=4$ and, equivalently, as a subvariety of $\Hilb^6(\P(B^\vee))$, $\dim B=3$; see \cite[Sections 5--6]{Sch01}. 

For $d\geq 1$, let 
\[
\bH_d(X)\coloneqq \Hilb_{dt+1}(X)
\] 
be the Hilbert scheme of curves $C\subset X$ with Hilbert polynomial $\chi(\cO_C(tH))=dt+1$. In contrast to the general behavior of Hilbert schemes, it is known that $\bH_d(X)$ only parametrizes Cohen--Macaulay (CM) curves in degrees $d\leq 6$ (\cite[Lemma 3.2]{BF14}). We identify a connected component of the Hilbert scheme of the expected dimension $\chi(N_{C/X})=d$, whose general point parametrizes a smooth rational curve. Hence, for simplicity, we use the same notation $\bH_d(X)$ for both the entire Hilbert scheme and the connected component under consideration.

The Hilbert schemes in degrees $d\leq 3$ have been studied by several authors \cite{KS04, AF06, Sch01, KPS18, She10, Kuz97}. More precisely, $\bH_1(X)$ is a (smooth) plane quartic curve by \cite[Lemma 3.1]{AF06}, $\bH_2(X)\cong \P^2$ by \cite[Proposition 2.3.6]{KPS18}, and $\bH_3(X)\cong \P^3$ by \cite[Theorem 2.4]{KS04}. Shen constructed the conic space and its universal family geometrically in \cite[Propositions A.1.6 and A.1.12]{She10}. Kuznetsov's flop construction \cite{Kuz97} and Faenzi's Beilinson spectral sequence \cite[Corollary 7.3]{Fae07} provided birational and homological descriptions. Our purpose is to extend this picture to degrees up to six and to determine the corresponding Hilbert schemes globally. 

Our main result is the following.
\begin{theorem}[\protect{Proposition \ref{thm-deg4-iso}, Proposition \ref{thm-deg5-iso}, and Proposition \ref{thm-deg6-iso}}]\label{thm-main}
Let $X$ be a general prime Fano threefold of degree $22$. Then the following statements hold.
\begin{enumerate}
\item There are isomorphisms $\bH_4(X)\cong \Gr(2,A)\cong \Gr(2,4)$.
\item There is an isomorphism $\bH_5(X)\cong \Gr_{\P^2}(3,\mathcal{R}_5)$, where $\mathcal{R}_5$ is a vector bundle of rank $4$ on $\P^2\cong \bH_2(X)$.
\item There is an isomorphism $\bH_6(X)\cong \mathrm{Bl}_{\Gamma_X}N$, where $N=\mathbf{K}_3(3,2)$ is the moduli space of stable three-arrow Kronecker representations of dimension vector $(3,2)$ and $\Gamma_X\cong \Delta_X\cong \bH_1(X)$ is the smooth plane quartic curve.
\end{enumerate}
\end{theorem}

The three descriptions arise from a common construction associated with the strong exceptional collection $\Db(X)=\langle E,K,U,\mc O_X\rangle$, constructed and studied in \cite{Kuz97}, \cite[Section 2]{AF06}, and \cite[Sections 6--7]{Fae07}; see Subsection \ref{subsec-V22} for the exceptional bundles and their basic properties. The relation between the bundle $K$ and the universal family of twisted cubics is given in Remark \ref{rem:K-Fourier-Mukai}. 
The Beilinson spectral sequence of \cite[Corollary 7.3]{Fae07} suggests that the ideal sheaf $\cI_{C_d}$ of a general rational curve $C_d$ of degree $d\geq 4$ is the middle cohomology of a complex:
\begin{equation}\label{mon}
0\longrightarrow (d-2)E\longrightarrow (d-3)K\longrightarrow (d-4)U\longrightarrow 0.
\end{equation}
For degrees $4\leq d\leq6$, we prove that the kernel of the last morphism is a slope-stable vector bundle and that the first morphism is fiberwise injective; see Sections \ref{sec-v22}, \ref{sec-degree-five} and \ref{sec-degree-six}, respectively. Its cokernel is then the ideal sheaf of a Cohen--Macaulay curve with Hilbert polynomial $dt+1$. The uniform construction used in these three cases is developed in Proposition \ref{prop-inj}, Corollary \ref{cor-exact}, and Proposition \ref{prop-family}. 
On the other hand, the three algebraic descriptions in Theorem \ref{thm-main} are accompanied by geometric interpretations. Recall that under Schreyer's realization of $X$, a point $x\in X$ corresponds to a twisted cubic $\Gamma_x\subset \P(A^\vee)$ (\cite{Sch01}). In degree $d=4$, a point $[W]\in \Gr(2,A)$ determines a line 
\[
L_W\coloneqq \P((A/W)^\vee)\subset \P(A^\vee).
\] 
Then the quartic curve $C_{W}$ corresponding to $[W]$ under the isomorphism in Theorem \ref{thm-main} is the one-dimensional family of twisted cubics having $L_W$ as a bisecant line (For details, see Proposition \ref{bilinecubic}).

In the degree $d=5$ case, the natural projection map
\[
\rho_5\colon \bH_5(X) \cong \Gr_{\P^2}(3,\mathcal{R}_5)\longrightarrow \bH_2(X)\cong \P^2
\]
of the relative Grassmannian associates to a general rational quintic curve its unique bisecant conic. For details, see Proposition \ref{prop-ubconic}. Moreover, it is well known that for a smooth conic $Q\subset X$, the Sarkisov link centered at $Q$ gives a smooth quadric threefold $Y_Q\subset \P^4$ \cite[Theorem 2.2 and Remark 2.10]{KP18}. A quintic curve parameterized by a point in the fiber $\rho_5^{-1}([Q])$ has strict transform of degree one with respect to the morphism to $Y_Q$, and is therefore mapped to a line.  

In the degree $d=6$ case, the isomorphism in Theorem \ref{thm-main} associates a general rational sextic curve with its three bisecant conics (Lemma \ref{lem-num-bisec}). For details, see Proposition \ref{prop-sextic-bisecants}. The degree six case has a further consequence. The blow-up $\bH_6(X)\to N$ allows one to recover the blow-up center $\Gamma_X\subset N$ (Corollary \ref{cor-blowup}). The determinantal embedding of $\Gamma_X$ recovers the discriminant quartic curve $\Delta_X$ by Proposition \ref{thm;blowup-center}. Combining this with Schreyer's result (\cite[Theorem 1.1 and Corollary 1.2]{Sch01}), we obtain the following.
\begin{corollary}\label{coro:torelli-H6}
Let $X$ and $X'$ be general prime Fano threefolds of genus $12$.
If $\bH_6(X)\cong \bH_6(X')$, then $X\cong X'$.
\end{corollary}
Lastly, the explicit descriptions of the moduli spaces and their universal families also permit the computation of DT-type invariants by Grothendieck--Riemann--Roch and intersection theory (\cite{CMT18, CT21, KP08}). Together with the previously known cases in degrees at most three \cite[Proposition 3.17 and Theorems 3.18 and 4.2]{CLW24}, our computations establish the DT-type invariants in degrees up to six. For detailed computations, see Section \ref{sec-dt1}.

A dimension calculation for the corresponding monads in \eqref{mon}, as in \cite[Proposition 2.6]{DP85}, agrees with the expected dimension of the Hilbert scheme of rational curves only for $d\leq 6$. This numerical restriction explains the range of the present construction and indicates why curves of higher degree require a different method.
\subsection{Organization of the paper}
The rest of this paper is organized as follows. Section \ref{sec-preliminaries} reviews the models of $X$, the exceptional bundles, Kronecker representations, and the determinantal constructions used later. Sections \ref{sec-v22}, \ref{sec-degree-five} and \ref{sec-degree-six} treat rational quartics, quintics, and sextics, respectively. Section \ref{sec-dt} contains the Torelli-type theorem and DT-type invariant computations from $\bH_6(X)$.

\subsection*{Acknowledgements}
This work was partially carried out while the first author was visiting Chuo University in Tokyo, and the first author gratefully acknowledges Fumiya Okamura for his hospitality.

\section{Preliminaries}\label{sec-preliminaries}
Throughout the paper, \(X\) is a general prime Fano threefold of degree $22$ over $\C$, with \(\Pic(X)=\Z H\), $-K_X=H$ and $H^3=22$. Let $\ell$ and $p$ be the classes of a line and a point satisfying $H^2=22\ell$ and $H\ell=p$.

\subsection{Models of the Fano threefold}\label{subsec-ABmu}
Fix vector spaces $A$ and $B$ with $\dim A=4$ and $\dim B=3$. The defining net of quadrics $q\colon B^\vee \longrightarrow S^2A^\vee$ induces a composition map
\begin{equation}\label{orm}
\mu\colon B^\vee \otimes A\longrightarrow A^\vee.
\end{equation}
For $b\in B^\vee$, define
\[
q_b\colon A\longrightarrow A^\vee,\quad  q_b(a)\coloneqq \mu(b\otimes a).
\] 
Each map $q_b$ is symmetric. For a general $V_{22}$, this net was studied in \cite[Sections 5--6]{Sch01}. We use the following two descriptions of $X$:
\begin{itemize}
\item the realization of $X$ as a family of twisted cubics in $\P(A^\vee)$ annihilated by the net $q$;
\item Mukai's tri-isotropic Grassmannian model $X\subset \Gr(3,7)$.
\end{itemize}
All the morphisms appearing below are induced by \(\mu\) in \eqref{orm} and by the evaluation morphisms associated with the exceptional collection. The relevant constructions are described in \cite{Sch01,AF06,Fae07}.

\subsection{Lines and conics}
Let
\[
\Delta_X\coloneqq \{[b]\in \P(B^\vee)\mid \det(q_b)=0\}
\]
be the plane quartic curve studied in \cite{AF06}. We first recall the description of the Hilbert scheme of lines and conics.

\begin{proposition}[\protect{\cite[Lemma 3.1]{AF06} and \cite[Theorem 6.1]{Sch01}}]\label{linesv22}
For a general $X$, the curve $\Delta_X$ is a smooth plane quartic and there is an isomorphism $\Delta_X\cong \bH_1(X)$. If $[b]\in \Delta_X$ corresponds to a line $L_b\subset X$, then $q_b$ has corank one and there is an exact sequence
\[
0\longrightarrow E\longrightarrow K\xrightarrow{b} U\longrightarrow \OO_{L_b}(-1)\longrightarrow 0.
\]
For $[b]\notin \Delta_X$, the morphism $b_X\colon K\to U$ is surjective.
\end{proposition}

\begin{proposition}[\protect{\cite[Lemma 3.2]{AF06}}]\label{conicsv22}
We have the following assertions.
\begin{enumerate}
\item There is an isomorphism $\Gr(2,B^\vee)\cong \P(B)$, and $\P(B)$ is the Hilbert scheme of conics on $X$.
\item There are at most six conics passing through each point of \(X\).
\end{enumerate}
\end{proposition}

For \(0\neq a\in A\), define \(\rho_a\colon B^\vee\to A^\vee\), \(\rho_a(b)\coloneqq q_b(a,-)\). For \([W]\in\Gr(2,A)\), define \(\mu_W\colon W\otimes B^\vee\to A^\vee\), \(\mu_W(w\otimes b)\coloneqq q_b(w,-)\).

The following result excludes linear subspaces contained in the discriminant quartic. It will be used in the construction of the Hilbert scheme of sextic rational curves.

\begin{lemma}\label{lem;degree-six-rank}
Let \(X\) be a prime Fano variety with a smooth plane quartic curve $\Delta_X$. Then
\begin{enumerate}
  \item for every \(0\neq a\in A\), one has \(\rk(\rho_a)\ge2\);
  \item for every \(W\in\Gr(2,A)\), one has \(\rk(\mu_W)\ge3\).
\end{enumerate}
\end{lemma}

\begin{proof}
For (1), suppose first that \(\rk(\rho_a)\le1\) for some \(0\neq a\in A\). Then
\(\dim\ker(\rho_a)\ge2\). For every \(b\in\ker(\rho_a)\), the equality
\(q_b(a,-)=0\) shows that \(q_b\) is singular. Hence, \(\P\bigl(\ker(\rho_a)\bigr)\subset\Delta_X\). The left-hand side contains a projective line, which is impossible since \(\Delta_X\) is a smooth plane quartic curve. This proves the first assertion.

For (2), suppose now that \(\rk(\mu_W)\le2\) for some \(W\in\Gr(2,A)\). Choose a
two-dimensional subspace \(T\subset\bigl(\operatorname{Im}\mu_W\bigr)^\perp\subset A\).
Then \(q_b(W,T)=0\) for every \(b\in B^\vee\). Set \(r\coloneqq\dim(W\cap T)\).

Assume first that \(r=0\). Then \(A=W\oplus T\), and each \(q_b\) is block
diagonal with respect to this decomposition:
\[
  q_b=
  \begin{pmatrix}
    Q_{1,b}&0\\
    0&Q_{2,b}
  \end{pmatrix}.
\]
Consequently, \(\det(q_b)=\det(Q_{1,b})\det(Q_{2,b})\). Thus, \(\Delta_X\) is reducible, contrary to its smoothness.

Assume that \(r=1\). Choose a basis \(e_1,e_2,e_3,e_4\) of \(A\) such that \(W=\langle e_1,e_2\rangle\) and \(T=\langle e_1,e_3\rangle\). The condition \(q_b(W,T)=0\) gives a matrix of the form
\[
  q_b=
  \begin{pmatrix}
    0&0&0&\ell_1(b)\\
    0&\ell_2(b)&0&*\\
    0&0&\ell_3(b)&*\\
    \ell_1(b)&*&*&*
  \end{pmatrix},
\]
where \(\ell_1,\ell_2,\ell_3\) are linear forms on \(B^\vee\). Hence, \(\det(q_b)=-\ell_1(b)^2\ell_2(b)\ell_3(b)\). Again, \(\Delta_X\) is reducible or nonreduced, a contradiction.

Finally, assume that \(r=2\). Then \(W=T\). Choose a decomposition
\(A=W\oplus W'\). Then
\[
  q_b=
  \begin{pmatrix}
    0&C_b\\
    C_b^{\mathsf t}&D_b
  \end{pmatrix}.
\]
It follows that \(\det(q_b)=\det(C_b)^2\), again contradicting the reducedness of $\Delta_X$. 
\end{proof}

The rank conditions above imply the following surjectivity of the composition and evaluation morphisms.

\begin{proposition}\label{prop-surj}
The composition $\mu: B^\vee \otimes A\lr A^\vee$ in \eqref{orm} is surjective. Moreover, for every subspace $S\subset \Hom(K,U)=B^\vee$ with $\dim S\geq 2$, the evaluation morphism $S\otimes K\lr U$ is surjective.
\end{proposition}
\begin{proof}
The assertion follows from \cite[Lemma 3.1]{AF06}. For a nonzero element $b\in B^\vee$, 
\[
\det(q_b)=0
\quad\Longleftrightarrow\quad
b_X\text{ is not surjective}
\quad\Longleftrightarrow\quad
[b]\in\Delta_X.
\]

Let $S\subset B^\vee$ be a subspace with $\dim S\geq 2$. Since
$\Delta_X\subset\PP(B^\vee)$ is a smooth plane quartic, it
contains no projective line. In particular, $\PP(S)\not\subset\Delta_X$.
Choose a nonzero element $b\in S$ such that $[b]\notin\Delta_X$.
Then $q_b\colon A\to A^\vee$ is an isomorphism. Since $q_b$ is the
restriction of the composition map
\[
\mu|_{S\otimes A}\colon S\otimes A\longrightarrow A^\vee
\]
to the subspace $\langle b\rangle\otimes A$, the map
$\mu|_{S\otimes A}$ is surjective. Taking $S=B^\vee$ proves the first assertion. The same choice of $b$ gives a surjective morphism $b_X\colon K\to U$. This morphism is the restriction of the evaluation map $S\otimes K\to U$
to $\langle b\rangle\otimes K$. The evaluation map is therefore
surjective, which proves the second assertion.
\end{proof}

The following lemma is helpful for understanding the forgetful morphisms in the degree $5$ and $6$ cases.

\begin{lemma}\label{lem-num-bisec}
The number $N_d$ of bisecant conics of a general rational curve of degree $d$ on $X$ is 
\[
N_{d}=\binom{d-3}{2}
\]
for $d\geq5$.
\end{lemma}
\begin{proof}
The proof is almost the same as that of \cite[Lemma 4.14]{FGP19}, hence we omit it. 
\end{proof}

\subsection{Exceptional bundles and stability}\label{subsec-V22}
Kuznetsov constructed a strong exceptional collection 
\[
\Db(X)=\langle E, U, Q^\vee, \OO_X\rangle
\] 
using the flop diagram relating two $\PP^1$-bundles over $X$ and $\PP^3$. Arrondo and Faenzi studied the dual exceptional collection 
\begin{equation}\label{dbaf}
\Db(X)=\langle E, K, U, \OO_X \rangle,
\end{equation}
see \cite{AF06, Fae07}. The bundles $E$, $K$, and $U$ are arithmetically Cohen--Macaulay, and the Hom-spaces are
\[
\Hom(E,K)=A, \quad \Hom(E,U)=A^\vee, \quad \Hom(K,U)=B^\vee.
\]
The natural composition is compatible with the defining net $\mu\colon B^\vee \otimes A\to A^\vee$ in \eqref{orm}. The evaluation morphisms give the exact sequences
\begin{equation}
\label{eq-af-one}
0\longrightarrow \wedge^2U\longrightarrow A\otimes E\stackrel{\mathrm{ev}}{\longrightarrow} K\longrightarrow 0,
\end{equation}
\begin{equation}
0\longrightarrow K\longrightarrow B\otimes U\longrightarrow Q^\vee\longrightarrow 0.
\end{equation}

\begin{remark}\label{rem:K-Fourier-Mukai}
By \cite[Theorem 2.4]{KS04}, the Hilbert scheme of twisted cubics
on $X$ is isomorphic to $\P(A^\vee)$. Under Schreyer's model \cite[Theorem 1.1 and Section 5]{Sch01}, the corresponding incidence variety $\mathcal C\subset \P(A^\vee)\times X$ can be viewed as the family of twisted cubics $\Gamma_x\subset\P(A^\vee)$ parametrized by $x\in X$. Let
\[
p\colon\mathcal C\longrightarrow X,
\qquad
q\colon\mathcal C\longrightarrow\P(A^\vee)
\]
be the projections such that $p^{-1}(x)=\Gamma_x$.

Consider the Fourier--Mukai transform
\[
\Phi_{\mathcal C}\colon\Db\bigl(\P(A^\vee)\bigr)\longrightarrow\Db(X),
\qquad
\Phi_{\mathcal C}(F)\coloneqq Rp_*\bigl(q^*F\otimes^{\mathbf L}\OO_{\mathcal C}\bigr)[1].
\]
The universal resolution of twisted cubic curves in $\PP(A^\vee)$ gives
\[
R^1p_*\bigl(\OO_{\mathcal C}\otimes q^*\OO_{\P(A^\vee)}(-2)\bigr)\cong\ker\bigl(A^\vee\otimes E\longrightarrow U\bigr).
\]
Dualizing \eqref{eq-af-one}, tensoring by $\OO_X(-1)$, and using $E^\vee(-1)\cong E$ and $(\wedge^2U)^\vee(-1)\cong U$, we obtain an exact sequence
\[
0\longrightarrow K^\vee(-1)
\longrightarrow A^\vee\otimes E
\longrightarrow U
\longrightarrow0.
\]
It follows that $\Phi_{\mathcal C}\bigl(\OO_{\P(A^\vee)}(-2)\bigr)\cong K^\vee(-1)$.
Equivalently, 
\[
K\cong
\Phi_{\mathcal C}
\bigl(\OO_{\P(A^\vee)}(-2)\bigr)^\vee(-1).
\]
Thus, the exceptional bundle $K$ can be recovered from the universal family of twisted cubics by this Fourier--Mukai functor.
\end{remark}

By \cite[Lemmas 6.1 and 6.9]{Fae07}, the Chern characters are
\[
\ch(E)=2-H+4\ell-\frac{1}{6}p, \quad \ch(K)=5-2H+4\ell+\frac{2}{3}p, \quad \ch(U)=3-H+\ell+\frac{1}{3}p.
\]
For a torsion-free sheaf $F\in\mathrm{Coh}(X)$ with $\rk (F)>0$, the \emph{slope} of $F$ is defined by
\[
\mu_H(F)= \frac{c_1(F)\cdot H^2}{\rk(F)H^3}.
\]
The sheaf $F$ is slope-stable (resp. slope-semistable) if every nonzero proper saturated subsheaf $G\subset F$ satisfies $\mu_H(G)<\mu_H(F)$ (resp. $\mu_H(G)\leq \mu_H(F)$). All the relevant bundles in the exceptional collection of $X$ are slope-stable \cite{AF06,Fae07}. Their slopes are 
\[
\mu_H(E)=-\frac{1}{2}, \quad \mu_H(U)=-\frac{1}{3}, \quad \mu_H(K)=-\frac{2}{5}, \quad \mu_H(\wedge^2 U)=-\frac{2}{3}.
\]

We will use the following property of saturated subsheaves of a direct sum of a stable bundle.
\begin{lemma}[{\cite[Lemma 1.2.13 and Corollaries 1.2.8 and 1.6.11]{HL}}]\label{lem;degree-six-polystable}
Let \(G\) be a slope-stable vector bundle and let \(T\) be a finite-dimensional
vector space. If \(F\subset T\otimes G\) is saturated and
\(\mu_H(F)=\mu_H(G)\), then there is a vector subspace \(T'\subset T\) such
that \(F=T'\otimes G\).
\end{lemma}

Semiorthogonality gives the following vanishings.
\[
\Ext^{>0}(E,E)=\Ext^{>0}(E,K)=\Ext^{>0}(E,U)=0, \quad \RHom(K,E)=\RHom(U,E)=0.
\]
A sheaf $F$ is called \emph{acyclic} if $R\Gamma(X,F)=0$.

\subsection{Kronecker representations and universal families}
A representation of the three-arrow Kronecker quiver with dimension vector $(m,n)$ is a linear map
\[
\beta\colon V_m\longrightarrow V_n\otimes B^\vee.
\]
A subrepresentation is a pair $V_m'\subset V_m$, $V_n'\subset V_n$ such that $\beta(V_m')\subset V_n'\otimes B^\vee$. For dimension vector $(3,2)$, we use the convention in which $\beta$ is stable when every nonzero proper subrepresentation $(V_3',V_2')$ satisfies $2\dim V_3'-3\dim V_2'<0$.

Let $N=\mathbf{K}_3(3,2)$ be the corresponding moduli space of stable three-arrow Kronecker representations. Since $\gcd(3,2)=1$, stability and semistability agree for the chosen weight. The space $N$ is a smooth Fano sixfold with
$\Pic(N)=\Z H_N$, $-K_N=3H_N$, where $H_N$ is the ample generator of $\Pic(N)$ (\cite{Dre88}). We fix universal bundles $\mathcal{V}_3$ (resp. $\mathcal{V}_2$) of ranks $3$ (resp. $2$) on the space $N$, together with a universal representation (see \cite[Sections 4--5]{Kin94})
\[
  \pmb{\beta}\colon \mathcal{V}_3\longrightarrow \mathcal{V}_2\otimes B^\vee.
\]
Moreover, the maximal-minor locus of a general Kronecker matrix defines a length-three subscheme of $\P(B)$. That is, there exists a birational map (cf. \cite[Theorem 4]{Dre88})
\[
N\dashrightarrow \Hilb^3(\P(B)).
\]
We use this map on the open locus where the length-three subscheme is reduced. 

\subsection{Determinantal loci}\label{subsec-determinantal-criterion}
We recall the codimension-two determinantal construction for a morphism of vector bundles whose ranks differ by one.

\begin{theorem}[\protect{\cite[Theorems A2.55, A2.60 and Examples A2.67, A2.68]{Eis05}}]\label{thm-det-codim-two}
Let \(Y\) be a smooth projective variety, and let \(\alpha\colon F\to G\) be a morphism of vector bundles with \(\rk(F)+1=\rk(G)\). Assume that the maximal-minor locus $D(\alpha)\subset Y$ has codimension $2$. Then 
\begin{enumerate}
\item The subscheme $D(\alpha)$ is Cohen--Macaulay of pure codimension two.
\item The morphism $\alpha$ is injective as a morphism of sheaves and there exists an exact sequence
\begin{equation}
\label{eq-det-seq}
0\longrightarrow F\stackrel{\alpha}{\longrightarrow}G\longrightarrow I_{D(\alpha)}\otimes \det(G)\otimes \det(F)^\vee\longrightarrow 0.
\end{equation}
\end{enumerate}
\end{theorem}

The following corollary gives a criterion for the maximal-minor locus to have codimension two. 

\begin{corollary}\label{cor-det-codim}
  Let $Y$ be a smooth projective variety, and let $\alpha\colon F\to G$ be a morphism of vector bundles with $\rk(F)+1=\rk(G)$. Let $D(\alpha)$ be its maximal-minor degeneracy subscheme. Assume the following.
  \begin{enumerate}[(i)]
    \item $\alpha$ has generic maximal rank,
    \item $D(\alpha)$ is nonempty, and
    \item the support of $D(\alpha)$ contains no divisorial component.
  \end{enumerate}
  Then $D(\alpha)$ has codimension $2$.
\end{corollary}
\begin{proof}
  Since $\alpha$ has generic maximal rank, the degeneracy locus $D(\alpha)$ is a proper closed subscheme. The determinantal height bound gives
  \[
  \codim_Y D(\alpha)\leq \rk(G)-\rk(F)+1=2
  \]
  for every irreducible component of $D(\alpha)$; see \cite[Theorem A2.54]{Eis05}. Assumption (iii) excludes components of codimension one. Consequently, every irreducible component has codimension two.
\end{proof}

We now prove the injectivity of the evaluation morphisms used in the degrees $4,5$ and $6$ cases.

\begin{proposition}\label{prop-inj}
  Let $X$ be a general $V_{22}$, and $n\in \{2,3,4\}$. Let $M$ be a slope-semistable vector bundle satisfying $\rk(M)=2n+1$ and $c_1(M)=-nH$. For an $n$-plane $W\subset \Hom(E,M)$, let $\mathrm{ev}_W\colon W\otimes E\to M$ be the evaluation morphism. Assume that $\ch_2(M)-n\ch_2(E)\neq 0$. Then
  \begin{enumerate}[(1)]
    \item the map $\mathrm{ev}_W$ is injective as a morphism of sheaves,
    \item the maximal-minor locus is Cohen--Macaulay of pure codimension $2$, and
    \item there is an exact sequence 
    \begin{equation}
\label{eq-det-seq-form}
0\longrightarrow W\otimes E\longrightarrow M\longrightarrow I_{\mathcal{C}_W}\otimes \det(M)\otimes \det(E)^{-n}\otimes \det(W)^{-2}\longrightarrow 0,
\end{equation}
where $\mathcal{C}_W$ is the maximal-minor degeneracy curve.
  \end{enumerate}
\end{proposition}
\begin{proof}
  Let $F\coloneqq \ker(\mathrm{ev}_W)$, $r=\rk(F)>0$, and $c_1(F)=aH$. If $r=2n$, then every element of $W$ induces the zero morphism $E\to M$, contradicting $W\subset \Hom(E,M)$. Hence, $r<2n$. Since the image is torsion-free, $F$ is saturated in $W\otimes E$ and $\frac{a}{r}\leq -\frac{1}{2}$.  The image is a proper subsheaf of $M$ of rank $2n-r$ with first Chern class $-(n+a)H$. Thus, we obtain
  \[
  \frac{-n-a}{2n-r}\leq -\frac{n}{2n+1},\quad a\geq -\frac{n(r+1)}{2n+1}.
  \]
  The integrality forces $r$ even and $a=-\frac{r}{2}$. By Lemma \ref{lem;degree-six-polystable}, $F=W_0\otimes E$ for a nonzero $W_0\subset W$, contradicting the definition of $W$ as a subspace of $\Hom(E,M)$. Hence, $\mathrm{ev}_W$ is injective. Let $I$ be the image and $\tilde{I}$ its saturation in $M$. If the rank-drop locus has a divisorial component, then $c_1(\tilde{I})=(-n+m)H$ for some $m\geq 1$. Since $M$ is semistable, we have
  \[
  \frac{-n+m}{2n}\leq -\frac{n}{2n+1},
  \]
  which is impossible. If the degeneracy locus is empty, then the cokernel is a line bundle with the first Chern class zero and vanishing second Chern character, which is also a contradiction. 

  Consequently, the degeneracy locus is nonempty and has no divisorial component. Corollary \ref{cor-det-codim} implies that it has pure codimension two. Theorem \ref{thm-det-codim-two} shows that the degeneracy locus is Cohen--Macaulay and yields the exact sequence \eqref{eq-det-seq-form}.
\end{proof}

We apply Proposition \ref{prop-inj} to the classes $(n-1)[K]-(n-2)[U]$.
\begin{corollary}\label{cor-exact}
  Let $X$ be a general $V_{22}$, and $n\in \{2,3,4\}$. Let $M$ be a $\mu_H$-semistable vector bundle satisfying 
  \[
  [M]=(n-1)[K]-(n-2)[U]\quad \text{in}\quad K_0(X).
  \]
  For every $n$-plane $W\subset \Hom(E,M)$, there is an exact sequence
  \[
  0\lr W\otimes E\lr M\lr \mathcal{I}_{\mathcal{C}_W}\lr 0
  \]
  up to a constant one-dimensional factor. Moreover, the curve $\mathcal{C}_W$ is Cohen--Macaulay and
  \[
  \ch(\mathcal{I}_{\mathcal{C}_W})=1-(n+2)\ell+\frac{n}{2}p, \quad \chi(\OO_{\mathcal{C}_W}(tH))=(n+2)t+1.
  \]
\end{corollary}
\begin{proof}
  The numerical assumptions on $M$ imply that $\rk(M)=2n+1$, $c_1(M)=-nH$, and $\ch_2(M)-n\ch_2(E)\neq 0$. Hence, one can apply Proposition \ref{prop-inj}. Moreover, since $\det(M)\cong \det(E)^n$, the determinant factor in \eqref{eq-det-seq-form} is pulled back from the parameter space of $W$ and is a fixed one-dimensional factor on each fiber. Finally,
  \begin{align*}
    \ch(I_{\mc C_W})&=(n-1)\ch(K)-(n-2)\ch(U)-n\ch(E)\\
    &=1-(n+2)\ell+\frac{n}{2}p.
  \end{align*}
  Using $\mathrm{td}(X)=1+\frac{1}{2}H+\frac{23}{6}\ell+p$, the Riemann--Roch formula gives
  \[
  \chi(\OO_{\mc C_W}(tH))=(n+2)t+1.\qedhere
  \]
\end{proof}

The construction in Proposition \ref{prop-inj} can be extended to families over a smooth base.

\begin{proposition}\label{prop-family}
  Let $T$ be a smooth variety, let $p_T\colon T\times X\to T$ and $p_X\colon T\times X\to X$ be the projections, and let $\mathcal{M}$ be a vector bundle on $T\times X$. Let $\mathcal{W}$ be a vector bundle of rank $n$ on $T$, and let 
  \[
  \mathbf{u}\colon p_T^{\ast}\mathcal{W}\otimes p_X^{\ast}E\longrightarrow \mathcal{M}
  \]
  be a morphism whose restriction to every closed fiber satisfies Proposition \ref{prop-inj}. Assume
  \[
  \det(\mathcal{M})\otimes \det(p_T^{\ast}\mathcal{W}\otimes p_X^{\ast}E)^\vee\cong p_T^{\ast}\Lambda
  \]
  for a line bundle $\Lambda$ on $T$. Let $\mathcal{C}\subset T\times X$ be the maximal-minor locus of the map $\mathbf{u}$. Then $\mathcal{C}$ is Cohen--Macaulay of pure codimension $2$ and is flat over $T$. Moreover, there is an exact sequence
  \[
  0\lr p_T^{\ast}\mathcal{W}\otimes p_X^{\ast}E\lr \mathcal{M}\lr \mathcal{I}_{\mathcal{C}}\otimes p_T^{\ast}\Lambda\lr 0.
  \]
\end{proposition}
\begin{proof}
  Every closed fiber of $\mc C\to T$ is a curve of codimension two in $X$ by Proposition \ref{prop-inj}. Hence, every irreducible component of $\mc C$ has codimension two in $T\times X$. The determinantal height bound gives the opposite inequality, and hence, $\mc C$ has pure codimension two. Applying Theorem \ref{thm-det-codim-two} to $T\times X$ gives the exact sequence and shows that $\mc C$ is Cohen--Macaulay. 

  The morphism $\mc C\to T$ is equidimensional of relative dimension one, and $\mc C$ is Cohen--Macaulay and the base is smooth. Miracle flatness therefore implies that $\mc C$ is flat over $T$; see \cite[Theorem 18.16]{Eis95}.
\end{proof}

\subsection{Ideal sheaves and deformation theory}
In this subsection, we first recall that an ideal sheaf of codimension at least two is simple.
\begin{lemma}\label{lem-endo}
  Let $X$ be an integral projective variety with $H^0(X,\OO_X)=\C$, and let $\mathcal{I}\subset \OO_X$ be the ideal sheaf of a subscheme of codimension at least two. Then $\mathrm{End}(\mathcal{I})=\C$.
\end{lemma}
\begin{proof}
  The sheaf $\mc I$ is torsion-free of rank one and agrees with $\OO_X$ in codimension one. Since $X$ is normal, $\mc I^{\vee\vee}\cong \OO_X$. Every endomorphism of $\mc I$ extends uniquely to $\mc I^{\vee\vee}$ and is therefore multiplication by a global regular function on $X$. Since $H^0(X,\OO_X)=\C$, the assertion follows.
\end{proof}

The following lemma recovers the subspace $W$ from the quotient map $M\to I$.
\begin{lemma}\label{lem-recover W}
  Suppose
  \[
  0\longrightarrow W\otimes E\longrightarrow M\xrightarrow{q} \mathcal{I}\longrightarrow 0
  \]
  is exact, $E$ is an exceptional bundle, and $q$ is fixed up to scalar. Then
  \[
  W=\ker(\Hom(E,M)\longrightarrow \Hom(E,I)).
  \]
\end{lemma}
\begin{proof}
  Apply $\Hom(E,-)$ to the given exact sequence. The exceptionality of $E$ gives
  \[
  0\longrightarrow W\longrightarrow \Hom(E,M)\longrightarrow \Hom(E,\cI)\longrightarrow 0,
  \]
  which proves the assertion.
\end{proof}

We will also use the following results to compare the parameter spaces with the Hilbert scheme.

\begin{proposition}[cf. \protect{\cite[(2.12.1) and (2.14.2), Ch. I]{Kol96}}]\label{prop-Hilb-tan}
Let $X$ be a smooth projective threefold satisfying $H^1(X,\mathcal O_X)=0$, and let $C\subset X$ be a Cohen--Macaulay curve with ideal sheaf $I_C$. Then $\dim T_{[C]}\Hilb(X)=\mathrm{ext}^1(I_C,I_C)$. Moreover, if
\[
\mathrm{ext}^1(I_C,I_C)\leq \dim_{[C]}\Hilb(X),
\]
then $\Hilb(X)$ is smooth at $[C]$.
\end{proposition}

\begin{proposition}[{\cite[Lemma 2.1]{GK19}}]\label{prop-rank-two}
  Let $R$ be a rank-two reflexive sheaf on a smooth projective threefold $X$. Then
  \[
  \deg c_3(R)=h^0(\mathcal{E}xt^1_X(R,\OO_X))\geq 0.
  \]
  In particular, $c_3(R)=0$ if and only if $R$ is locally free.
\end{proposition}

For a rank-two sheaf $R$ on $X$ with $c_1(R)=-H$, the Hirzebruch--Riemann--Roch formula gives 
\begin{equation}\label{eq-rank-two}
  \chi(R)=\frac{1}{2}c_3(R).
\end{equation}

Faenzi describes the Beilinson spectral sequence associated with the exceptional collection $\langle E, K, U, \OO_X\rangle$ in \cite[Corollary 7.3]{Fae07}. For the ideal sheaves considered here, the vanishings reduce it to the complexes in Table \ref{tab:rateulr}. In degrees $4,5$ and $6$, the complexes are realized below by evaluation morphisms and extend over the corresponding projective parameter spaces.

\begin{table}[htbp]
\centering
\begin{tabular}{|c||m{8cm}|}
\hline
Degree of curve $C_d$ & Resolution (or complex) of the ideal sheaf $I_{C_d}$ on $X$ \\
\hline
\hline
$d=3$ & $0\lr E \lr U \lr I_{C_3}\lr 0$ \\
\hline
$d\geq 4$ & $0\to (d-2)E\to (d-3)K\to (d-4)U\to 0$ \\
\hline
\end{tabular}
\caption{Beilinson complexes for general rational curves on $X$}
\label{tab:rateulr}
\end{table}

The cases $d=1,2$ are described by the corresponding degeneracy loci; see \cite[Lemmas 3.1 and 3.2]{AF06}. For $4\leq d \leq 6$, these complexes can be treated using the exceptional collection.

We recall an incidence description of a blow-up, which will be applied to the degree six case.

\begin{lemma}[\protect{\cite[Proposition 6.2]{CLS22}}]\label{lem-blowup}
  Let $L\colon \cA\to \cB$ be a morphism of vector bundles of ranks $b+r$ and $b$, respectively, on a smooth variety $T$. Let $Z\coloneqq D_{b-1}(L)$ be the degeneracy locus with the maximal-minor scheme structure. Assume that $Z$ is smooth of codimension $r+1$ and that $D_{b-2}=\emptyset$. Let $p\colon \Gr_T(r,\cA)\to T$ be the projection map and $\cW\subset p^{\ast}\cA$ be tautological subbundle. Then the zero scheme
  \[
  Y_L\coloneqq Z(\cW\longrightarrow p^{\ast}\cA\xrightarrow{p^{\ast}L}\cB)
  \]
  is isomorphic over $T$ to $\mathrm{Bl}_Z T$.
\end{lemma}

\section{The degree-four case}\label{sec-v22}
In this section, we describe the Hilbert scheme of rational quartic curves on $X$, extending the result of \cite{CKK25}. 
\subsection{The Hilbert scheme of rational quartics}
Set $G_4\coloneqq \Gr(2,A)$. For a two-dimensional subspace $W\subset A=\Hom(E,K)$, let 
\[
u_W\colon W\otimes E\longrightarrow K
\]
be the evaluation morphism.

\begin{proposition}\label{prop-deg4-inj}
  For every $W\in G_4$, the morphism $u_W$ is injective and there is an exact sequence
  \begin{equation}\label{eq-deg4}
    0\longrightarrow W\otimes E\xrightarrow{u_W} K\longrightarrow \mathcal{I}_{\mathcal{C}_W}\longrightarrow 0,
  \end{equation}
  where $\mathcal{C}_W\subset X$ is a Cohen--Macaulay curve with arithmetic genus zero.
\end{proposition}
\begin{proof}
  This is immediate from Corollary \ref{cor-exact}.
\end{proof}

Let $\mathcal{W}_4\subset A\otimes \OO_{G_4}$ be the tautological bundle of rank two, and let $p_4\colon G_4\times X\to G_4$ and $q_4\colon G_4\times X\to X$ be the projections. Consider the universal evaluation morphism
\[
\mathbf{u}_4\colon p_4^{\ast}\mathcal{W}_4\otimes q_4^{\ast}E\longrightarrow q_4^{\ast}K.
\]
Its maximal-minor locus is $\cC_4$, and Theorem \ref{thm-det-codim-two} gives an exact sequence
\begin{equation}\label{univ:4}
  0\longrightarrow p_4^{\ast}\mathcal{W}_4\otimes q_4^{\ast}E\longrightarrow q_4^{\ast}K\longrightarrow \mathcal{I}_{\mathcal{C}_4}\otimes p_4^{\ast}(\det \mathcal{W}_4)^{-2}\longrightarrow 0.
\end{equation}
Here, we use $\det K\cong (\det E)^2$. Proposition \ref{prop-family} gives a flat family and a morphism
\begin{equation}\label{p4map}
  \Psi_4\colon G_4\longrightarrow \bH_4(X),\quad W\mapsto [C_W].
\end{equation}
\begin{proposition}\label{thm-deg4-iso}
The morphism $\Psi_4$ in \eqref{p4map}  is an isomorphism.
\end{proposition}
\begin{proof}
  Applying $\Hom(K,-)$ to \eqref{eq-deg4} and using $\RHom(K,E)=0$ gives
  \[
  \Hom(K,K)\xrightarrow{\sim} \Hom(K,\mathcal{I}_{\mathcal{C}_W}).
  \]
  Since $K$ is exceptional, the quotient $K\to \mathcal{I}_{\mathcal{C}_W}$ is unique up to scalar. Lemma \ref{lem-recover W} therefore recovers $W=\ker\left(A\to \Hom(E,\mathcal{I}_{\mathcal{C}_W})\right)$ uniquely. Hence, $\Psi_4$ is injective on closed points and thus $4\leq \dim_{[\mathcal{C}_W]}\bH_4(X)$. On the  other hand, applying $\RHom(E,-)$ and $\RHom(K,-)$ to \eqref{eq-deg4}, and using that the exceptional collection is strong, we obtain
  \[
  \Hom(E,\mathcal{I}_{\mathcal{C}_W})=A/W,\quad \Ext^{>0}(E,\mathcal{I}_{\mathcal{C}_W})=0,\quad \RHom(K,\mathcal{I}_{\mathcal{C}_W})=\C.
  \]
  Moreover, Lemma \ref{lem-endo} shows that the map $\Hom(\mathcal{I}_{\mathcal{C}_W},\mathcal{I}_{\mathcal{C}_W})\to \Hom(K,\mathcal{I}_{\mathcal{C}_W})$ is an isomorphism. By applying $\RHom(-,\mathcal{I}_{\mathcal{C}_W})$ to \eqref{eq-deg4} again, we obtain
  \[
  \Ext^1(\mathcal{I}_{\mathcal{C}_W},\mathcal{I}_{\mathcal{C}_W})\cong \Hom(W\otimes E,I_{\cC_W})\cong W^\vee \otimes (A/W).
  \]
In particular, $\bH_4(X)$ is smooth at $[\cC_W]$ by Proposition \ref{prop-Hilb-tan}. Since the Grassmannian $G_4$ is smooth, we can conclude that the map $\Psi_{4}$ is an isomorphism by Zariski's Main theorem.
\end{proof}

\begin{remark}
The proof of Proposition \ref{thm-deg4-iso} uses only the properties of the exceptional bundles $E,K,U$, their slope-stability, and the exact sequence \eqref{eq-af-one}; it does not use the smoothness of $\Delta_X$. Since these properties remain valid for the Mukai--Umemura threefold, the same argument applies. 
\end{remark}

The stability of the universal subbundle $U$, together with \cite[Lemma 3.1]{BF14}, gives the following vanishing.
\begin{lemma}[\protect{\cite[Lemma 3.1]{BF14}}]\label{lem:qv}
Let $C\subset X$ be a smooth rational curve of degree $d\geq 4$. Then 
\[ 
H^0(C,Q^\vee|_C)=0. 
\] 
\end{lemma}

\begin{definition}
A morphism $f\colon \PP^1\lr X$ is called \emph{free} if $f^*T_X$ is globally generated.
\end{definition}

\begin{proposition}\label{exi4}
A general point of $G_4=\Gr(2,A)\cong \Gr(2,4)$ parametrizes a smooth rational quartic curve on $X$.
\end{proposition}
\begin{proof}
By \cite[Theorem 1.2]{BJ22}, there exists a free smooth rational quartic curve $C(\cong\PP^1)$ on $X$ lying in a component whose evaluation map has connected fibers. We claim that the ideal sheaf $I_C$ has the resolution in \eqref{eq-deg4}.
By the Grauert--M\"ulich theorem (\cite[Proposition 3.1]{PRT24}), the restrictions of the exceptional bundles split as
\[
E|_C\cong \cO_{\PP^1}(-2)^{\oplus 2},
\quad
U|_C\cong
\cO_{\PP^1}(-2)\oplus\cO_{\PP^1}(-1)^{\oplus 2}.
\]
In particular, $H^0(C,E|_C)=H^0(C,U|_C)=0$. Also, Lemma \ref{lem:qv} gives $H^0(C,Q^\vee|_C)=0$. Hence, the Riemann--Roch theorem yields
\begin{equation}\label{eq:hc}
h^1(C,E|_C)=2,\quad h^1(C,U|_C)=1,\quad h^1(C,Q^\vee|_C)=0.
\end{equation}
Applying the Beilinson spectral sequence of \cite[Corollary 7.3]{Fae07} to $I_C$, we obtain a complex
\[
0\lr H^1(C,E|_C)\otimes E
\longrightarrow
H^1(C,U|_C)\otimes K
\longrightarrow
H^1(C,Q^\vee|_C)\otimes U,
\lr0\]
whose middle cohomology is $I_C$. By \eqref{eq:hc}, this complex reduces to an exact sequence
\[
0\longrightarrow V_2\otimes E
\xrightarrow{\alpha_C}
K
\longrightarrow
\cI_C
\longrightarrow 0, \qquad V_2\coloneqq H^1(C, E|_C)
\]
By Lemma \ref{lem-recover W}, we have $V_2\subset \Hom(E, K)=A$ and hence determines a point $[V_2]\in\Gr(2,A)=G_4$.
Thus, the space $G_4$ contains a point parametrizing a smooth rational quartic. Since $G_4$ is irreducible and its universal family in \eqref{univ:4} is flat, the locus where the fibers are smooth is open and nonempty. Therefore, a general fiber is a smooth rational quartic.
\end{proof}

\subsection{Bisecant lines and rational quartics}
Let $\iota\colon \wedge^2U\lr A\otimes E$ be the first morphism in \eqref{eq-af-one}, and let $\pi_W\colon A\lr A/W$ be the canonical projection map. Combining this with the second map in \eqref{eq-deg4}, we have another locally free resolution of $I_{C_W/X}$ 
\begin{equation}\label{eq:wr}
0\lr \wedge^2 U \stackrel{(\Phi \coloneqq \pi_W \otimes \mathrm{id}_{E})\circ\iota}{\xrightarrow{\quad}}(A/W)\otimes E\lr I_{C_W/X}\lr 0.
\end{equation}
The morphism $\Phi$ is induced by the first syzygies of the corresponding twisted cubics (for details, see \cite[Section 3]{Fae07}).
Using the exact sequence \ref{eq:wr} and a classical result of Vainsencher and Xavier \cite{VX02}, one can obtain the geometric meaning of the isomorphism $\Psi_4$ in Proposition \ref{thm-deg4-iso}.

\begin{proposition}\label{bilinecubic} 
For a general point $[W]\in \Gr(2,A)$, let 
\[ 
L_{W} \coloneqq \PP((A/W)^\vee)\subset \PP(A^\vee) 
\] 
be the corresponding line. Then 
\[ 
C_W= \left\{ x\in X \mid \length(\Gamma_x\cap L_W)\ge 2 \right\}, 
\]
where $\Gamma_x\subset\PP(A^\vee)$ is the twisted cubic parametrized by $x\in X$. In particular, $C_W$ parametrizes the twisted cubics for which $L_W$ is a bisecant line. \end{proposition}
\begin{proof} 
By \eqref{eq:wr}, a point $x \in X$ belongs to $C_W$ precisely when the fiber morphism 
\[ 
\wedge^2U_x\longrightarrow (A/W)\otimes E_x 
\] 
has a nonzero kernel. Since $\dim U_x=3$, every nonzero element of $\wedge^2U_x$ determines a pencil $\Pi\subset U_x$. By the construction of the map $\Phi$ in \eqref{eq:wr}, this is equivalent to the residual line of $\Pi$ being $L_W$. By the classical residual description of twisted cubics, this is equivalent to $L_W$ being a bisecant line of $\Gamma_x$ (\cite{VX02}). Hence, the assertion follows. 
\end{proof}

\section{The degree-five case}\label{sec-degree-five}
Recall that $B^\vee=\Hom(K, U)$. Let $S\subset B^\vee$ be a two-plane and 
\[
\phi_S\colon S\otimes K\longrightarrow U, \quad M_S\coloneqq \ker(\phi_S).
\]

\subsection{The Hilbert scheme of rational quintics}
Let $S\subset B^\vee$ be a two-dimensional subspace. We first study the kernel $M_S$ of $S\otimes K\to U$.
\begin{lemma}\label{MSst}
For every $S\in \Gr(2,B^\vee)$, the following assertions hold:
\begin{enumerate}
\item The morphism $\phi_S$ is surjective,
\item The sheaf $M_S$ is a slope-stable vector bundle of rank seven and $c_1(M_S)=-3H$,
\item The map $\mu_S\colon A\otimes S\lr A^\vee$ is surjective, and 
      \[
      R_S\coloneqq \Hom(E,M_S)=\ker(\mu_S)
      \]
      has dimension four. Moreover, $\Ext^1(E,M_S)=0$.
\end{enumerate}
\end{lemma}
\begin{proof}
Part (1) follows from Proposition \ref{prop-surj}. For (2), recall that $\rk(K)=5, c_1(K)=-2H$ and $\rk(U)=3, c_1(U)=-H$. Let $F\subset M_S$ be a nonzero saturated destabilizing subsheaf with $\rk(F)=r$, $1\leq r\leq 6$ and $c_1(F)=m H$. Then
\[
\mu(M_S)=-\frac{3}{7}\leq \mu(F)=\frac{m}{r}\leq -\frac{2}{5}=\mu(S\otimes K).
\]
The only possibility is $(r, m)=(5, -2)$, thus, $\mu(F)=-\frac{2}{5}$. However, all nonzero Jordan--H\"{o}lder factors of $S\otimes K$ are $K$ and thus $F\cong K\otimes \langle v\rangle$ for some nonzero $v\in S$. Hence, this implies that the corresponding map $v\in \Hom(K, U)$ is zero since $F$ is contained in the kernel bundle $M_S$. This contradicts the fact that $\dim S=2$.

For (3), choose $0\neq b\in S$ with $[b]\not\in \Delta_X$. The restriction of $\mu_S$ to $\langle b\rangle \otimes A$ is $q_b$, hence is an isomorphism. Applying $\Hom(E, -)$ to the short exact sequence $\ses{M_S}{S\otimes K}{U}$ gives
\begin{equation}\label{eq-deg5}
  0\longrightarrow \Hom(E, M_S)\longrightarrow S\otimes A\xrightarrow{\mu_S} A^\vee\longrightarrow \Ext^1(E, M_S)\longrightarrow 0,
\end{equation}
where the vanishing of the last term follows from $\Ext^1(E,S\otimes K)=0$. The map $\mu_S$ is surjective by Proposition \ref{prop-surj}. Consequently, we have $\Ext^1(E,M_S)=0$ and $\dim\Hom(E, M_S)=\dim\ker(\mu_S)=\dim(A\otimes S)-\dim A^\vee=2\times 4-4=4$.
\end{proof}

For a three-plane $W\subset R_S$, Corollary \ref{cor-exact} gives 
\begin{equation}\label{eq-deg5-eval}
  0\longrightarrow W\otimes E\longrightarrow M_S\longrightarrow \mathcal{I}_{\mathcal{C}_{S,W}}\longrightarrow 0,
\end{equation}
where $\mathcal{C}_{S,W}$ is a Cohen--Macaulay curve of degree five and arithmetic genus zero.

Let $P_5\coloneqq \Gr(2,B^\vee)$ and $\mathcal{S}_5\subset B^\vee \otimes \OO_{P_5}$ be the tautological bundle. On $P_5\times X$, define
\[
0\longrightarrow \mathcal{M}_5\longrightarrow \mathcal{S}_5\boxtimes K\longrightarrow U\longrightarrow 0.
\]
By \eqref{eq-deg5}, the kernels of the universal maps form a rank-four vector bundle $\mathcal{R}_5$ on $P_5$, fitting into the exact sequence
\[
0\longrightarrow \mathcal{R}_5\longrightarrow \mathcal{S}_5\otimes A\longrightarrow A^\vee\otimes \OO_{P_5}\longrightarrow 0.
\]
In particular, $\mathcal{R}_5$ has rank four. Let $\pi_5\colon G_5\coloneqq \Gr_{P_5}(3,\mathcal{R}_5)\to P_5$ and let $\mathcal{W}_5\subset \pi_5^{\ast}\mathcal{R}_5$ be the tautological rank-three bundle. Let $p_5\colon G_5\times X\to G_5$ and $q_5\colon G_5\times X\to X$ be the projections. The universal evaluation morphism is
\[
\mathbf{u}_5\colon \mathcal{W}_5\boxtimes E\to \pi_5^{\ast}\mathcal{M}_5.
\]
Note that we have
\[
\det(\pi_5^{\ast}\mathcal{M}_5)\otimes \det(\mathcal{W}_5\boxtimes E)^\vee \cong p^{\ast}_{5}\Lambda_5,
\]
where $\Lambda_5=\pi_5^{\ast}(\det \mathcal{S}_5)^5\otimes (\det \mathcal{W}_5)^{-2}$.
Proposition \ref{prop-family} gives a flat family and a morphism 
\begin{equation}\label{pi5m}
\Psi_5\colon G_5\longrightarrow \bH_5(X),\quad (S,W)\mapsto [C_{(S,W)}].
\end{equation}
We next show that the ideal sheaf $I_{C_{S,W}}$ determines both $S$ and $W$.

\begin{lemma}\label{lem-deg5-inj}
  The ideal sheaf $I=\mathcal{I}_{C_{S,W}}$ determines the pair $(S,W)$. In particular, $\Psi_5$ is injective on closed points.
\end{lemma}
\begin{proof}
  Applying $\RHom(U,-)$ to the exact sequences $\ses{M_S}{S\otimes K}{U}$ and $\ses{W\otimes E}{M_S}{I}$, and using $\RHom(U,E)=\RHom(U,K)=0$, we obtain $\Ext^1(U,I)\cong \C$. Let
  \[
  0\longrightarrow I\longrightarrow P_C\longrightarrow U\longrightarrow 0
  \]
  be the unique nonzero extension. Then, for the pair $(S,W)$, we have $P_C\cong (S\otimes K)/(W\otimes E)$. Moreover, this extension is nonsplit. Indeed, the resolution $\ses{W\otimes E}{S\otimes K}{P_C}$ and the semiorthogonality give $\Hom(U,P_C)=0$. Applying $\Hom(K,-)$ to the resolution gives $\Hom(K,P_C)\cong S$. The image of the map induced by $P_C\to U$,
  \[
  \Hom(K,P_C)\longrightarrow \Hom(K,U)=B^\vee
  \]
  is precisely the two-plane $S$. Hence, $S$ is recovered from $I$.

  Once $S$ has been recovered, the bundle $M_S$ is determined. The semiorthogonality gives $\RHom(M_S,E)=0$ and therefore
  \[
  \Hom(M_S,I)\cong \mathrm{End}(M_S)=\C.
  \]
  By Lemma \ref{lem-recover W}, we conclude the proof.
\end{proof}

\begin{proposition}\label{thm-deg5-iso}
  The morphism $\Psi_5$ in \eqref{pi5m} is an isomorphism.
\end{proposition}
\begin{proof}
  By Lemma \ref{lem-deg5-inj}, $\Psi_5$ is injective on closed points. On the other hand, one can easily see that by its construction, $\Ext^1(M_S,M_S)\cong \Hom(S,B^\vee/S)$ and $\Ext^i(M_S,M_S)=0$ for $i\geq 2$. By applying $\RHom(E,-)$ to \eqref{eq-deg5-eval}, we obtain
  \[
  \Hom(E,I)\cong R_S/W,\quad \Ext^{>0}(E,I)=0.
  \]
  Since $\RHom(M_S,E)=0$, we also have
  \[
  \RHom(M_S,I)\cong \RHom(M_S,M_S).
  \]
  By Lemma \ref{lem-endo}, the map $\Hom(I,I)\to \Hom(M_S,I)$ is an isomorphism. Applying $\RHom(-,I)$ to \eqref{eq-deg5-eval}, we obtain the exact sequence
  \[
  0\longrightarrow \Hom(W,R_S/W)\longrightarrow \Ext^1(I,I)\longrightarrow \Hom(S,B^\vee/S)\longrightarrow 0.
  \]
In particular, $\dim\Ext^1(I,I)=5$. By the same argument as in Proposition \ref{thm-deg4-iso}, the morphism $\Psi_{5}$ is an isomorphism.
\end{proof}

\begin{proposition}\label{exi5}
A general point of $G_5$ parametrizes a smooth rational quintic curve on $X$.
\end{proposition}
\begin{proof}
Let $C\subset X$ be a general free smooth rational quintic curve. The same vanishing arguments as in the proof of Proposition \ref{exi4}, together with the Riemann--Roch theorem, reduce the Beilinson spectral sequence to a complex
\begin{equation}\label{exis5}
0\lr V_3\otimes E
\stackrel{\alpha}{\longrightarrow}
V_2\otimes K
\stackrel{\beta}{\longrightarrow}
V_1\otimes U\lr0,\quad \dim V_i=i,
\end{equation}
whose middle cohomology is $I_C$. The morphism $\beta$ is determined by a linear map
\[
V_2\longrightarrow B^\vee=\Hom(K,U).
\]
This map is injective. Indeed, otherwise one obtains a contradiction to the injectivity of the map $\alpha$ in \eqref{exis5}. Therefore, this curve $C$ is parametrized by $G_5$, which completes the proof.
\end{proof}

\subsection{Bisecant conics and the forgetful morphism}\label{subsec-1conic}
A general rational quintic on $X$ has a unique bisecant conic (Lemma \ref{lem-num-bisec}). We identify this conic with the image under the natural projection
\[
\rho_5\colon \bH_5(X)\cong G_5\lr P_5\cong \bH_2(X)\cong \PP^2.
\]

\begin{proposition}\label{prop-ubconic}
Let $[C]\in \rho_5^{-1}([Q])$, where $C\subset X$ is a smooth rational quintic and $Q\subset X$ is a smooth conic. Then $Q$ is a bisecant conic to $C$. For a general point $[C]\in \bH_5(X)$, the conic $Q$ is the unique bisecant conic to $C$.
\end{proposition}
\begin{proof}
Write $C=C_{S,W}$, and let $Q$ be the conic corresponding to $S$. By \cite[Lemma 3.2]{AF06}, we have a natural isomorphism \(P_5\coloneqq\Gr(2,B^\vee)\cong \P(B)\), and $\P(B)$ is naturally identified with the Hilbert scheme \(\bH_2(X)\) of conics on \(X\). For \(S\in P_5\), let \(\lambda_S\coloneqq S^\perp\subset B\).

The evaluation morphism \(\lambda_S\otimes U\to \mc Q^\vee\) has degeneracy locus \(Q\)
and fits into an exact sequence
\[
  0\longrightarrow \lambda_S\otimes U
  \longrightarrow \mc Q^\vee
  \longrightarrow I_{Q}
  \longrightarrow0.
\]
Comparing this sequence with
\[
  0\longrightarrow K
  \longrightarrow B\otimes U
  \longrightarrow \mc Q^\vee
  \longrightarrow0
\]
and using \(B/\lambda_S\cong S^\vee\), we obtain
\begin{equation}\label{star1}
  0\longrightarrow K
  \longrightarrow S^\vee\otimes U
  \longrightarrow I_{Q}
  \longrightarrow0.
\end{equation}
Dualizing \eqref{star1} with respect to $\omega_X$ gives
\[
0 \longrightarrow 
\mathcal{H}om_X(I_{Q},\omega_X)
\longrightarrow 
S\otimes U^\vee(-1)
\longrightarrow 
K^\vee(-1)
\longrightarrow 
\mathcal{E}xt_X^1(I_{Q},\omega_X)
\longrightarrow 0.
\]
Since $Q$ is Cohen--Macaulay of codimension two, we have
\[
\mathcal{H}om_X(I_{Q},\omega_X)\cong \omega_X, \quad
\mathcal{E}xt_X^1(I_{Q},\omega_X)
\cong
\mathcal{E}xt_X^2(\mathcal{O}_{Q},\omega_X)
\cong
\omega_{Q}.
\]
Moreover, since $\omega_X\cong\mathcal{O}_X(-1)$ and $U^\vee(-1)\cong \wedge^2 U$,
we obtain
\begin{equation}\label{star3}
  0\longrightarrow\OO_X(-1)
  \longrightarrow S\otimes\wedge^2U
  \xrightarrow{\delta_S}
  K^\vee(-1)
  \longrightarrow\omega_{Q}
  \longrightarrow0.
\end{equation}
The map $\delta_S$ is the dual of the inclusion map $K\lr S^\vee\otimes U$. Equivalently, under the natural isomorphism $B^\vee=\Hom(\wedge^2 U, K^\vee(-1))$, it is the restriction to $S\subset B^\vee$ of the universal evaluation map $B^\vee\otimes\wedge^2 U\lr K^\vee(-1)$.

Let \(\mu_S=\mu|_{S\otimes A}\colon S\otimes A\longrightarrow A^\vee\) be the restriction of the map \(\mu\colon B^\vee\otimes A\to A^\vee\) in \eqref{orm}, and set \(R_S\coloneqq\ker(\mu_S)\). The map \(\mu_S\) is surjective by the proof of Proposition \ref{prop-surj}. In particular, \(\dim R_S=4\). Dualizing and twisting \eqref{eq-af-one} gives
\[
  0\longrightarrow K^\vee(-1)
  \longrightarrow A^\vee\otimes E
  \longrightarrow U
  \longrightarrow0.
\]
The evaluation morphisms fit into the commutative diagram
\[
\begin{array}{ccccccccc}
0
&\longrightarrow&
S\otimes\wedge^2U
&\longrightarrow&
S\otimes A\otimes E
&\longrightarrow&
S\otimes K
&\longrightarrow&
0
\\
&&
\big\downarrow{\delta_S}
&&
\big\downarrow{\mu_S\otimes\mathrm{id}_E}
&&
\big\downarrow{\varphi_S}
\\
0
&\longrightarrow&
K^\vee(-1)
&\longrightarrow&
A^\vee\otimes E
&\longrightarrow&
U
&\longrightarrow&
0.
\end{array}
\]
The kernels of the middle and right vertical morphisms are
\(R_S\otimes E\) and \(M_S\), respectively. The snake lemma and
\eqref{star3} give
\[
  0\longrightarrow\OO_X(-1)
  \longrightarrow R_S\otimes E
  \longrightarrow M_S
  \longrightarrow\omega_{Q}
  \longrightarrow0.
\]
For $[W]\in \Gr(3,R_S)$, the induced map $W\otimes E\lr M_S$ is injective. Applying the snake lemma again, we have an exact sequence 
\[
\ses{J\coloneqq \mathrm{coker}\{\cO_X(-1)\lr (R_S/W)\otimes E\}}{I_{C}}{\omega_{Q}}.
\]
One can see that $J=J_{C\cup Q}$ and thus, $I_{C\cap Q/Q}\cong \omega_{Q}=\cO_{Q}(-1)$. In fact, $J$ is the ideal sheaf of a curve of degree $7$ and genus $1$ in $X$. Therefore, the conic $Q$ is bisecant to $C$.
\end{proof}

Let \(Q\subset X\) be a smooth conic. The Sarkisov link centered at $Q$ constructed by Kuznetsov--Prokhorov has the form $X\dashrightarrow Y_Q$, where \(Y_Q\subset \mathbf P^4\) is a smooth quadric threefold; see \cite{KP18}. If $\sigma\colon \widetilde{X}=\mathrm{Bl}_Q(X)\to X$ is the blow-up and $E_Q$ is the exceptional divisor, the induced morphism $\widetilde{X}\to Y_Q$ is defined by $|\sigma^{\ast}H-2E_Q|$. For a general $C\in\rho_5^{-1}([Q])$, 
Proposition \ref{prop-ubconic} shows that \(Q\) is bisecant to \(C\). Hence, for the strict transform $\widetilde{C}\subset \widetilde{X}$ of $C$, we have $(\sigma^{\ast}H-2E_Q)\cdot\widetilde C = 5-2\cdot2 = 1$. Therefore, the Sarkisov link maps $\widetilde{C}$ onto a line in $Y_Q$. The next proposition describes both $Y_Q$ and the fiber of $\rho_5$ in terms of $R_S$.

\begin{proposition}\label{geme}
Let $[S]\in P_5$ correspond to a smooth conic \(Q=Q_S\subset X\).
Then there is a unique line $\CC\cdot\omega_Q\subset\wedge^2R_S^\vee$
spanned by a nondegenerate alternating form such that
\[
Y_Q\cong \mathrm{LG}(2,R_S,\omega_Q).
\]
Furthermore, the form $\omega_Q$ induces natural isomorphisms
\[
\rho_5^{-1}([Q])=\Gr(3,R_S)\xrightarrow{\sim}\Gr(1,R_S)\xrightarrow{\sim}\bH_1(Y_Q).
\]
\end{proposition}

\begin{proof}
Applying \(\Hom(E,-)\) to \eqref{star1}, we have $R_S^\vee\cong \operatorname{Hom}(E,I_Q)$.
Since \(\det E\cong \cO_X(-H)\), taking the determinant of a pair of morphisms $E\to I_Q$ defines a linear map
\begin{equation}\label{eq:451}
\lambda_Q\colon \wedge^2R_S^\vee \longrightarrow H^0(X,I_Q^2(H)).
\end{equation}
The linear system $|I_Q^2(H)|$ is the one defining the Sarkisov link to $Y_Q\subset \P^4$, and $h^0(X,I_Q^2(H))=5$. Since $\dim \wedge^2 R_S^\vee=6$, the kernel of $\lambda_Q$ is nonzero, and it is one-dimensional. Indeed, if $\dim\ker(\lambda_Q)\geq 2$, then the line \(\PP(\ker\lambda_Q)\subset \P(\wedge^2 R_S^\vee)\) meets the Pl\"{u}cker quadric $\Gr(2,R_S^\vee)$. A decomposable element in the kernel would give two linearly independent morphisms $E\to I_Q$ whose determinant vanishes identically, contrary to the stability of $E^\vee$. Therefore, $\ker(\lambda_Q)=\CC\cdot\omega_Q$.
The form $\omega_Q$ is non-decomposable and hence nondegenerate since $\dim R_S=4$. By dualizing \eqref{eq:451}, we obtain
\[
H^0(X,I_Q^2(H))^\vee=
\ker (
\wedge^2R_S\xrightarrow{\;\omega_Q\;}\CC
).
\]
Consequently,
\[
\mathrm{LG}(2,R_S,\omega_Q)=\Gr(2,R_S)
\cap
\PP\bigl(H^0(X,I_Q^2(H))^\vee\bigr)
\subset\PP^4
\]
is the smooth quadric threefold. By construction, this quadric is precisely \(Y_Q\), proving the claim. Since $\bH_1(Y_Q)\cong \PP(R_S)$, the latter claim clearly holds.
\end{proof}

\begin{remark}
Let $L\subset Y_Q$ be a general line. Its proper transform under the inverse Sarkisov link is a smooth rational quintic on $X$. The resolution obtained from the link and the Serre construction has the form \eqref{eq-deg5-eval}. Hence, this quintic determines a point of $G_5$.
\end{remark}

\section{The degree-six case}\label{sec-degree-six}
\subsection{Stable Kronecker representations and kernel bundles}

Let \(V_3\cong\C^3\) and \(V_2\cong\C^2\). A three-arrow Kronecker representation of dimension vector \((3,2)\) is a linear map
\[
  \beta\colon V_3\longrightarrow V_2\otimes B^\vee.
\]
We use the convention that \(\beta\) is stable if every nonzero proper
subrepresentation \((V_3',V_2')\) satisfies \(2\dim V_3'-3\dim V_2'<0\). This is King stability with the chosen sign convention; see \cite{Kin94}.

The same tensor defines the dual representation
\[
  \gamma\colon V_2^\vee\longrightarrow V_3^\vee\otimes B^\vee.
\]

The next proposition shows that the induced morphism $\beta_X$ is surjective.

\begin{proposition}\label{thm;degree-six-surj}
Let \(\beta\) be a stable three-arrow Kronecker representation of dimension
vector \((3,2)\). Then the induced morphism
\[
  \beta_X\colon V_3\otimes K\longrightarrow V_2\otimes U
\]
is surjective.
\end{proposition}

\begin{proof}
  The annihilator construction identifies subrepresentations
$(C,D)$ of the dual representation with subrepresentations
$(D^\perp,C^\perp)$ of $\beta$, and
\[
2\dim D^\perp-3\dim C^\perp
=
3\dim C-2\dim D.
\]
It is enough to prove that the dual morphism
\[
f\colon V_2^\vee\otimes U^\vee\longrightarrow V_3^\vee\otimes K^\vee
\]
is an inclusion of vector bundles. Let \(I\coloneqq\im(f)\), \(s\coloneqq\rk(I)\) and \(c_1(I)=mH\). The case $s=0$ contradicts the stability of $\gamma$. Since $I$ is a torsion-free quotient of $V_2^\vee\otimes U^\vee$ and a subsheaf of $V_3^\vee\otimes K^\vee$, we have
\[
\frac{1}{3}\leq \frac{m}{s}\leq \frac{2}{5}.
\]
For $1\leq s\leq 6$, the only possibilities are
\[
(s,m)=(3,1),\quad (5,2),\quad (6,2).
\]

Assume first that \((s,m)=(3,1)\). Then the kernel \(F\coloneqq\ker(f)\) is a saturated rank-three subsheaf of $V_2^\vee\otimes U^\vee$ of slope $\frac{1}{3}$. By Lemma \ref{lem;degree-six-polystable}, there is a line \(C\subset V_2^\vee\) such that \(F=C\otimes U^\vee\). Thus, \(f\) vanishes on \(C\otimes U^\vee\), and consequently \(\gamma(C)=0\). The pair \((C,0)\) is a subrepresentation of dimension vector \((1,0)\), whose weight is \(3\cdot1-2\cdot0=3>0\), which is a contradiction.

Assume next that \((s,m)=(5,2)\). Let \(\overline I\subset V_3^\vee\otimes K^\vee\) be the saturation of \(I\). Write \(c_1(\overline I)=(2+d)H\) for some \(d\ge0\). Since \(\overline I\) is a rank-five subsheaf of the semistable bundle \(V_3^\vee\otimes K^\vee\), \(\frac{2+d}{5}\le\frac25\) and hence, \(d=0\). By Lemma \ref{lem;degree-six-polystable}, there is a line \(D\subset V_3^\vee\) such that \(\overline I=D\otimes K^\vee\). Since \(I\subset\overline I\), the morphism \(f\) factors through \(D\otimes K^\vee\). Therefore, \(\gamma(V_2^\vee)\subset D\otimes B^\vee\). The pair \((V_2^\vee,D)\) is a subrepresentation of dimension vector \((2,1)\), whose weight is \(3\cdot2-2\cdot1=4>0\). This is again a contradiction.

Consequently, \((s,m)=(6,2)\) and \(f\) is injective as a morphism of sheaves.

Set \(C_f\coloneqq \mathrm{coker}(f)\). We prove that \(C_f\) is locally free.
Suppose, on the contrary, that it is not locally free. At some $x\in X$, the fiber map has a nonzero tensor $\xi\in V_2^\vee\otimes U_x^\vee$ in its kernel. Assume first that \(\xi\) has rank \(1\), i.e., \(\xi=c\otimes u^\vee\). Define
\[
  R_{x,u^\vee}\coloneqq\ker\left(B^\vee\longrightarrow K_x^\vee, \quad b\longmapsto b_x^\vee(u^\vee)\right).
\]
Every nonzero element of this space represents a map $K\to U$ dropping rank at $x$, hence a point of the smooth plane quartic $\Delta_X$. Since $\Delta_X$ contains no line, $\dim R_{x,u^\vee}\leq 1$. The equality \(f_x(\xi)=0\) then gives a subrepresentation of $\gamma$ of dimension $(1,0)$ or $(1,1)$, each of which gives a contradiction.

Assume now that \(\xi\) has tensor rank \(2\) and write
\[
\xi=c_1\otimes u_1^\vee+c_2\otimes u_2^\vee,\quad T=\langle u_1^\vee,u_2^\vee\rangle.
\]
Choose $0\neq u\in U_x$ such that $T=u^{\perp}$, and choose a basis \(d_1,d_2,d_3\) of \(V_3^\vee\). Write
\[
  \gamma(c_i)=\sum_{j=1}^3d_j\otimes b_{ji}.
\]
The equality \(f_x(\xi)=0\) implies that \(r_j\coloneqq b_{j1}\otimes u_1^\vee+b_{j2}\otimes u_2^\vee\) belongs to
\[
  R_x(T)\coloneqq \mc Q_x\cap(B^\vee\otimes T)
\]
for \(j=1,2,3\).

If \(\dim R_x(T)\geq 3\), then the dual map
\[
  B\longrightarrow\mc Q_x^\vee,
  \qquad
  b\longmapsto b_x(u),
\]
has a kernel of dimension at least \(2\), producing a projective line of conics through $x$, contrary to \cite[Lemma 3.2]{AF06}. Hence, \(\dim R_x(T)\le2\). The equation $f_x(\xi)=0$ says that $r_j=b_{j1}\otimes u_1^\vee+b_{j2}\otimes u_2^\vee$ belongs to $R_x(T)$. Since this space has dimension at most two, the three $r_j$ are linearly dependent. Independence of $u_1^\vee,u_2^\vee$ then gives a two-plane $D\subset V_3^\vee$ satisfying $\gamma(V_2^\vee)\subset D\otimes B^\vee$. This gives a subrepresentation of dimension $(2,2)$ which contradicts the stability of $\beta$. Thus, the cokernel $C_f$ of $f$ is locally free, and dualizing gives the proposition.
\end{proof}

We define the kernel bundle $M_{\beta}$ associated with a stable Kronecker representation.
\begin{definition}\label{def;M-beta}
For a stable Kronecker representation \(\beta\), let $\beta_X\colon V_3\otimes K\to V_2\otimes U$ be the induced morphism and set
\[
  M_\beta\coloneqq\ker\left(V_3\otimes K\xrightarrow{\beta_X}V_2\otimes U
  \right).
\]
\end{definition}
Then 
\begin{equation}\label{eq-deg6}
  0\longrightarrow M_{\beta}\longrightarrow V_3\otimes K\xrightarrow{\beta_X} V_2\otimes U\longrightarrow 0
\end{equation}
is exact, and $\rk(M_{\beta})=9$, $c_1(M_{\beta})=-4H$, $\ch(M_{\beta})=9-4H+10\ell+\frac{4}{3}p$. The object $M_{\beta}[1]$ belongs to the triangulated subcategory $\langle K,U\rangle$. Semiorthogonality gives $\RHom(M_{\beta},E)=0$, and in particular, $\Ext^1(M_{\beta},E)=0$.

\begin{lemma}\label{lem-rank-seven}
  Assume that $F\subset M_{\beta}$ is a saturated maximal destabilizing subsheaf with $
  \rk(F)=7$ and $c_1(F)=-3H$. Then there exist
  \begin{enumerate}[(1)]
    \item a slope-stable acyclic rank-two vector bundle $R$ with $c_1(R)=-H$, $c_2(R)=d'\ell$ and $8\leq d'\leq 15$,
    \item a nonzero morphism $M_{\beta}\to R$, and
    \item vector spaces $I,C$ with $\dim I=d'-7$ and $\dim C=d'-8$, and a monad
      \begin{equation}\label{eq-monad}
        C^\vee \otimes E\longrightarrow I^\vee\otimes K\xrightarrow{\gamma_X} I\otimes U,
      \end{equation}
          whose cohomology is $R$.
  \end{enumerate}
  Let $N_{\gamma}=\ker(\gamma_X)$. Then we have an exact sequence
  \begin{equation}\label{eq-N-gamma}
    0\longrightarrow C^\vee \otimes E\longrightarrow N_{\gamma}\longrightarrow R\longrightarrow 0.
  \end{equation}
\end{lemma}
\begin{proof}
  The quotient $G\coloneqq M_{\beta}/F$ is torsion-free of rank two and $c_1(G)=-H$, and is slope-stable by the maximality of $F$. 

  Let $c_2(G)=d\ell$. Since $\ch_2(M_{\beta})=10\ell$ and $\ch_2(G)=(11-d)\ell$, we obtain $\ch_2(F)=(d-1)\ell$. By the Bogomolov inequality for the semistable sheaf $F$, we have
  \[
  0\leq H(c_1(F)^2-2\rk(F)\ch_2(F))=212-14d,
  \]
  thus, $d\leq 15$.

  The exact sequence \eqref{eq-deg6} and the acyclicity of $K$ and $U$ imply that $M_{\beta}$ is also acyclic. Semistability and Serre duality give 
  \[
  H^0(F)=H^3(F)=H^0(G)=H^3(G)=0.
  \]
  The sequence $\ses{F}{M_{\beta}}{G}$ gives $H^1(F)=H^2(G)=0$.

  Let $R=G^{\vee\vee}$ and $T=R/G$. Then $R$ is stable and reflexive, and $T$ is supported in dimension at most one. Stability and the exact sequence $\ses{G}{R}{T}$ give $H^0(R)=H^2(R)=H^3(R)=0$. Hence, $\chi(R)=-h^1(R)\leq 0$. By \eqref{eq-rank-two} and Proposition \ref{prop-rank-two}, 
  \[
  0\leq \frac{1}{2}c_3(R)=\chi(R)\leq 0.
  \]
  Hence, $c_3(R)=0$ and $H^{\bullet}(X,R)=0$, and Proposition \ref{prop-rank-two} shows that $R$ is locally free.

  Write $c_2(R)=d'\ell$. Since the one-dimensional part of $T$ is effective, $d'\leq d\leq 15$. The bundle $R$ is an odd instanton in the sense of \cite{Fae14}. The minimal second Chern class is seven, and the unique odd instanton with $c_2=7\ell$ is $E$ by \cite[Subsection 4.2]{Fae14}. The composite $M_{\beta}\to G\hookrightarrow R$ is nonzero, whereas $\RHom(M_{\beta},E)=0$. Thus $R\not\cong E$, and $d'\geq 8$.
  
  By \cite[Proposition 4.3 and Lemma 4.5]{Fae14}, the bundle $R$ is the cohomology of a monad
\[
C^\vee\otimes E\longrightarrow I^\vee\otimes K\xrightarrow{\gamma_X}I\otimes U,
\]
where $\dim I=d'-7$ and $\dim C=d'-8$.
\end{proof}

For a nonzero Kronecker representation of dimension vector $(a,b)$ with $b>0$, set
\[
\mu_K(a,b)\coloneqq \frac{a}{b}.
\]
A Kronecker representation is semistable if every nonzero proper subrepresentation has slope at most the slope of the representation.

We first prove that the Kronecker representation in the instanton monad is semistable.

\begin{lemma}\label{lem-Hom-vanish}
The map $\gamma_X$ in \eqref{eq-monad} is induced by a semistable Kronecker representation 
  \[
  \gamma\colon I^\vee\to I\otimes B^\vee
  \]
of slope one.
\end{lemma}
\begin{proof}
  Suppose that $I_1\to I_2\otimes B^\vee$ is a subrepresentation with $a=\dim I_1>b=\dim I_2$. Let $I'$ and $N'$ be the image and kernel of the induced map $I_1\otimes K\longrightarrow I_2\otimes U$, and write $s=\rk(I')$ and $c_1(I')=mH$. If $s=0$, then $\mu_H(N')=-\frac{2}{5}>-\frac{1}{2}$. If $s>0$, then semistability gives
  \[
  -\frac{2}{5}\leq \frac{m}{s}\leq -\frac{1}{3},\quad s\leq 3b.
  \]
  Since $\rk(N')=5a-s$ and $c_1(N')=-(2a+m)H$, we obtain
  \[
  -2(2a+m)+(5a-s)=a-2m-s\geq a-\frac{s}{3}\geq a-b>0,
  \]
  which implies that $\mu_H(N')>-\frac{1}{2}$.

  The subrepresentation condition gives $N'\subset N_{\gamma}$. Let $K'=N'\cap(C^\vee\otimes E)$. Since $C^\vee\otimes E$ is polystable of slope $-\frac{1}{2}$, one has $\mu_H(K')\leq -\frac{1}{2}$. The composite $N'\to R$ cannot vanish. Its image $J\subset R$ therefore satisfies $\mu_H(J)>-\frac{1}{2}$. If $\rk(J)=1$, then this contradicts the stability of $R$. If $\rk(J)=2$, then $c_1(J)=c_1(R)-D$ for an effective divisor $D$; consequently, $\mu_H(J)\leq -\frac{1}{2}$, again a contradiction.
\end{proof}

We now prove that $M_{\beta}$ is slope-stable.

\begin{lemma}\label{lem-M-beta-stable}
If $\beta$ is stable, then $M_{\beta}$ is slope-stable.
\end{lemma}
\begin{proof}
  Suppose that $M_{\beta}$ is unstable, and let $F\subset M_{\beta}$ be its maximal destabilizing subsheaf. We may assume that $F$ is saturated and semistable. Write $r\coloneqq \rk(F)$ and $c_1(F)=mH$. Since $F\subset V_3\otimes K$, we have
  \[
  -\frac{4}{9}<\frac{m}{r}\leq -\frac{2}{5}.
  \]
  For $1\leq r\leq 8$, the only possibilities are
  \[
  (r,m)=(5,-2)\quad \text{or}\quad (7,-3).
  \]
  
  If $(r,m)=(5,-2)$, then $F$ is saturated in $V_3\otimes K$. Indeed, $(V_3\otimes K)/F$ is an extension of the torsion-free sheaf $M_{\beta}/F$ by the vector bundle $V_2\otimes U$. Since $F$ has the same slope as $K$, Lemma \ref{lem;degree-six-polystable} gives $F=D\otimes K$ for a line $D\subset V_3$. Since $F\subset M_{\beta}$, the restriction of $\beta_X$ to $D\otimes K$ vanishes and hence, $\beta(D)=0$. This gives a subrepresentation of dimension vector $(1,0)$ which contradicts Kronecker stability.

  If $(r,m)=(7,-3)$, then by Lemma \ref{lem-rank-seven}, there is a stable acyclic bundle $R$, a nonzero morphism $M_{\beta}\to R$, and a Kronecker representation $\gamma$ with kernel $N_{\gamma}$ fitting into \eqref{eq-N-gamma}. By Lemma \ref{lem-Hom-vanish}, $\gamma$ is semistable of slope one.

  Since $\Ext^1(M_{\beta},E)=0$, the morphism $M_{\beta}\to R$ lifts to $M_{\beta}\to N_{\gamma}$. Moreover, $M_{\beta}[1]$ and $N_{\gamma}[1]$ correspond to $\beta$ and $\gamma$. Hence, this lift induces a nonzero morphism of Kronecker representations $\beta\to \gamma$. Since there are no nonzero morphisms from a semistable Kronecker representation of larger slope to one of smaller slope, the lifted morphism $\beta\to\gamma$ is zero, a contradiction. Therefore, $M_{\beta}$ is stable.
\end{proof}

The exceptional pair $(K,U)$ also gives the deformation properties of $M_{\beta}$.

\begin{corollary}\label{cor:tn}
  For every stable $\beta$, we have $\mathrm{End}(M_{\beta})=\C$, $T_{[\beta]}N\cong \Ext^1(M_{\beta},M_{\beta})$, $\Ext^i(M_{\beta},M_{\beta})=0$ for all $i\geq 2$ and $\RHom(M_{\beta},E)=0$.
\end{corollary}
\begin{proof}
  The exact sequence \eqref{eq-deg6} realizes $M_{\beta}[1]$ as the image of $\beta$ under the derived equivalence for the exceptional pair $(K,U)$. The assertions follow from Kronecker stability and semiorthogonality.
\end{proof}

\subsection{The degeneracy locus $\Gamma_X$}
Applying $\Hom(E,-)$ to the sequence \eqref{eq-deg6} gives 
\begin{equation}\label{kcmap}
0\longrightarrow R_{\beta}\longrightarrow V_3\otimes A\xrightarrow{L_{\beta}} V_2\otimes A^\vee\longrightarrow C_{\beta}\longrightarrow 0,
\end{equation}
where $R_{\beta}=\Hom(E,M_{\beta})$ and $C_{\beta}=\Ext^1(E,M_{\beta})$. The universal map on $N$ is 
\[
\mathcal{L}\colon \mathcal{V}_3\otimes A\longrightarrow \mathcal{V}_2\otimes A^\vee.
\]
Applying $\Hom(E,-)$ to \eqref{eq-deg6}, we define the maximal-minor locus $\Gamma_X\subset N$.
\begin{definition}
  Define \(\Gamma_X\coloneqq D_7(\mathcal L)=V\bigl(I_8(\mathcal L)\bigr)\subset N\), where \(I_8(\mathcal L)\) is generated by the maximal minors of \(\mathcal L\).
\end{definition}

\begin{proposition}\label{thm;blowup-center}
There is a closed immersion \(j_X\colon\Delta_X\to N\) such that \(\Gamma_X=j_X(\Delta_X)\) scheme-theoretically. Consequently, \(\Gamma_X\cong\Delta_X\). Moreover, 
\[
  \rk L_\beta
  =
  \begin{cases}
    8,&[\beta]\notin\Gamma_X,\\
    7,&[\beta]\in\Gamma_X.
  \end{cases}
\]
The locus $\Gamma_X$ is smooth of codimension five, and
\[
N_{\Gamma_X/N}\cong \mathcal{H}om\left(\ker(\mathcal{L}|_{\Gamma_X}),\mathrm{coker}(\mathcal{L}|_{\Gamma_X})\right).
\]
\end{proposition}

\begin{proof}
We first construct $j_X$ in families. On $\Delta_X$, the universal symmetric map
\[
q_{\Delta}\colon A\otimes \OO_{\Delta_X}(-1)\longrightarrow A^\vee\otimes \OO_{\Delta_X}
\]
has constant rank three. Let $\mathcal{A}_{\Delta}\coloneqq \ker(q_{\Delta})\otimes \OO_{\Delta_X}(1)\subset A\otimes \OO_{\Delta_X}$. The morphism $B^\vee\otimes \mathcal{A}_{\Delta}\to A^\vee\otimes \OO_{\Delta_X}$ has rank two by Lemma \ref{lem;degree-six-rank}. Let $\mathcal{W}_{\Delta}\subset A\otimes \OO_{\Delta_X}$ be the annihilator of its image. Then $\mathcal{W}_{\Delta}$ has rank two, and
\[
\mathcal{V}_{3,\Delta}=\ker\left(\mathcal{W}_{\Delta}\otimes B^\vee\longrightarrow A^\vee\otimes \OO_{\Delta_X}\right)
\]
has rank three. The inclusion
\[
\mathcal{V}_{3,\Delta}\longrightarrow \mathcal{W}_{\Delta}\otimes B^\vee
\]
is a family of stable $(3,2)$-Kronecker representations. Indeed, for a line $\ell=\C\omega\subset (\mathcal{W}_{\Delta})_b$, we have \(V_{3,b}\cap(\ell\otimes B^\vee)=\ker(\rho_{\omega})\) which satisfies \(\dim\ker(\rho_{\omega})\le1\) by Lemma \ref{lem;degree-six-rank}.
Consequently, no destabilizing dimension vector occurs, and the family defines a morphism $j_X\colon \Delta_X\to N$. 

For $b\in \Delta_X$, write $W_b\coloneqq (\mathcal{W}_\Delta)_b$ and $\C{a_b}=(\cA_{\Delta})_b=\ker(q_b)$. The inclusion $(\mathcal{W}_{\Delta})_b\xhookrightarrow{} A$ gives a nonzero element of $\ker L^\vee_{j_X(b)}$, and thus, \(j_X(\Delta_X)\subset|\Gamma_X|\). 

Conversely, choose $[\beta]\in |\Gamma_X|$ and $0\neq \eta\in \ker L^\vee_{\beta}\subset V_2^\vee\otimes A$. The tensor $\eta$, viewed as $V_2\to A$, cannot have rank one. If \(\eta=\ell\otimes a\), then the row contraction $(\ell\otimes1)\beta\colon V_3\longrightarrow B^\vee$ has image in \(\ker(\rho_a)\), which has \(\dim\ker(\rho_a)\le1\). Its kernel would give a destabilizing subrepresentation of dimension at least $(2,1)$. Hence, the tensor $\eta$ has rank two and identifies $V_2$ with a two-plane $W\subset A$. The equation $L^\vee_{\beta}(\eta)=0$ gives 
\[
\im(\beta)\subset \ker(\mu_W).
\]
Both sides have dimension three by Lemma \ref{lem;degree-six-rank}. Choose $0\neq a\in (\im \mu_W)^{\perp}$. Then $W=(\im(\rho_a))^\perp$ and $\ker \rho_a=\C b$ for a unique $b\in \Delta_X$. It follows that $[\beta]=j_X(b)$. Hence, $j_X(\Delta_X)=|\Gamma_X|$.

At $j_X(b)$, the map $L$ has corank exactly one, since $\Delta_X$ cannot contain a projective line. Differentiating $L$ at $j_X(b)$ gives the normal map
\begin{equation}\label{numa}
\nu_b\colon T_{j_X(b)}N\longrightarrow \Hom(\ker L_{j_X(b)},\mathrm{coker} L_{j_X(b)}.)
\end{equation}
Since $\mu_{W_b}$ is surjective onto its image, a first-order deformation of the Kronecker map induces an arbitrary element of $\Hom(V_{3,b},\im \mu_{W_b})$. The only possible failure of surjectivity would come from 
\[
\ker L_{j_X(b)}\cap (V_{3,b}\otimes \C{a_b}).
\]
If $v\otimes a_b$ belongs to this intersection, then write $\beta_b(v)=w_1\otimes c_1+w_2\otimes c_2$. The equation $L_{\beta_b}(v\otimes a_b)=0$ gives $c_1,c_2\in \ker \rho_{a_b}=\C b$, and therefore, $\beta_b(v)=w\otimes b$. Since $\beta_b(v)\in \ker \mu_{W_b}$, one has $w\in \ker q_b=\C a_b$. Since $\Delta_X$ is smooth, the vector $a_b$ does not lie in $W_b$. Hence, $w=0$ and stability gives $v=0$, which implies $\nu_b$ is surjective.

The determinantal locus is therefore smooth of expected codimension five along its support. Thus, the finite bijection $j_X\colon \Delta_X\to |\Gamma_X|$ is in fact an isomorphism. In particular, $\Gamma_X$ is reduced. Since its support is $j_X(\Delta_X)$, the equality $\Gamma_X=j_X(\Delta_X)$ holds scheme-theoretically. The normal bundle follows from the differential of the maximal-minor equations.
\end{proof}

Let $p\colon \Gr_{N}(4,\mathcal{V}_3\otimes A)\to N$ be the relative Grassmannian, and let $\mathcal{W}_6$ be its tautological bundle of rank four. Define 
\begin{equation}\label{eq:relze}
G_6\coloneqq Z(\mathcal{W}_6\longrightarrow p^{\ast}(\mathcal{V}_3\otimes A)\xrightarrow{p^{\ast}\cL} p^{\ast}(\mathcal{V}_2\otimes A^\vee)).
\end{equation}
This is the vanishing locus of the composition map. Its closed points are the pairs $([\beta],W)$ satisfying $W\subset \ker L_{\beta}$. Let 
\begin{equation}\label{blmap}
  \pi_6\colon G_6\to N
\end{equation} 
be the restriction of $p$, and use the same notation $\mathcal{W}_6$ for the restricted tautological bundle.

\begin{corollary}\label{cor-blowup}
There is an isomorphism
  \[
  G_6\cong \mathrm{Bl}_{\Gamma_X}N.
  \]
  The fiber of $\pi_6$ is a point over $N\setminus \Gamma_X$ and is $\Gr(4,5)\cong \P^4$ over every point of $\Gamma_X$.
\end{corollary}
\begin{proof}
  The bundles in $\cL$ have ranks twelve and eight. By Lemma \ref{lem-blowup} ($b=8$ and $r=4$), $D_7(\cL)=\Gamma_X$ is smooth of codimension $5$ and $D_6(\cL)=\emptyset$. By Proposition \ref{thm;blowup-center}, we conclude the proof. 
\end{proof}

\subsection{The Hilbert scheme of rational sextics}
For $([\beta],W)\in G_6$, the inclusion $W\subset R_{\beta}=\Hom(E,M_{\beta})$ gives an evaluation morphism
\[
u_{\beta,W}\colon W\otimes E\longrightarrow M_{\beta}.
\]
By Lemma \ref{lem-M-beta-stable} and Corollary \ref{cor-exact}, we obtain
\begin{equation}\label{eq-I}
    0\longrightarrow W\otimes E\longrightarrow M_{\beta}\longrightarrow \mathcal{I}_{\mathcal{C}_{\beta,W}}\longrightarrow 0,
  \end{equation}
  where $\mathcal{C}_{\beta,W}$ is a Cohen--Macaulay curve of degree six and arithmetic genus zero. Let $\mathcal{M}_6$ be the pullback to $G_6\times X$ of the universal kernel on $N\times X$. The universal evaluation map is
\begin{equation}\label{eq:un6}
\mathbf{u}_6\colon \mathcal{W}_6\boxtimes E\longrightarrow \mathcal{M}_6.
\end{equation}
Let $\Lambda_6=(\det \mathcal{V}_3)^5\otimes (\det \mathcal{V}_2)^{-3}\otimes (\det \mathcal{W}_6)^{-2}$.
Proposition \ref{prop-family} gives a flat family and a morphism
\[
\Psi_6\colon G_6\longrightarrow \bH_6(X),\quad (\beta,W)\mapsto [C_{\beta,W}].
\]

We next show that $I_{C_{\beta,W}}$ determines $[\beta]$ and $W$. 

\begin{lemma}\label{lem-deg6-recover}
  The ideal sheaf $\mathcal{I}_{{\mathcal C}_{\beta,W}}$ determines $[\beta]\in N$ and $W\subset \Hom(E,M_{\beta})$.
\end{lemma}
\begin{proof}
  Applying $\RHom(-,E)$ to \eqref{eq-I} gives $\RHom(I_{C_{\beta,W}},E)\cong W^\vee[-1]$. Comparing the right mutation triangle of $I_{C_{\beta,W}}$ through $E$ with the triangle induced by \eqref{eq-I}, we obtain
\[
\mathbf R_E(I_{C_{\beta,W}})\cong M_\beta, \qquad
W\cong\Ext^1(I_{C_{\beta,W}},E)^\vee.
\]
The Kronecker equivalence then recovers $[\beta]$ from $M_\beta[1]$.
\end{proof}
\begin{proposition}\label{thm-deg6-iso}
  The morphism
  \[
  \Psi_6\colon G_6\longrightarrow \bH_6(X)
  \]
  is an isomorphism. Consequently, $\bH_6(X)\cong \mathrm{Bl}_{\Gamma_X}\mathbf{K}_3(3,2)$.
\end{proposition}
\begin{proof}
  By Lemma \ref{lem-deg6-recover}, $\Psi_6$ is injective on closed points. On the other hand,  at $([\beta],W)\in G_6$, let $\mathcal{C}_{\beta,W}=\Psi_6(\beta,W)$. Applying $\RHom(E,-)$ to \eqref{eq-I} gives
  \[
  \Hom(E,\cI_C)=R_{\beta}/W,\quad \Ext^1(E,\mathcal{I}_{{\mathcal C}_{\beta,W}})=C_{\beta},\quad \Ext^i(E,\mathcal{I}_{{\mathcal C}_{\beta,W}})=0  (i\geq 2).
  \]
  Moreover, since $\RHom(M_{\beta},E)=0$, we have $\RHom(M_{\beta},\cI_C)\cong \RHom(M_{\beta},M_{\beta})$. In particular, $\Hom(M_{\beta},\mathcal{I}_{{\mathcal C}_{\beta,W}})=\mathrm{End}(M_{\beta})=\C$. Lemma \ref{lem-endo} gives $\Hom(\mathcal{I}_{{\mathcal C}_{\beta,W}},\mathcal{I}_{{\mathcal C}_{\beta,W}})=\C$, and the map $\Hom(\mathcal{I}_{{\mathcal C}_{\beta,W}},\mathcal{I}_{{\mathcal C}_{\beta,W}})\to \Hom(M_{\beta},\mathcal{I}_{{\mathcal C}_{\beta,W}})$ is an isomorphism. Applying $\RHom(-,\mathcal{I}_{{\mathcal C}_{\beta,W}})$ to \eqref{eq-I} gives a long sequence:
  \begin{align}\label{del}
    0\longrightarrow \Hom(W,R_{\beta}/W)&\longrightarrow \Ext^1(\mathcal{I}_{{\mathcal C}_{\beta,W}},\mathcal{I}_{{\mathcal C}_{\beta,W}})\longrightarrow \Ext^1(M_{\beta},M_{\beta})\\
  &\stackrel{\delta}{\longrightarrow} \Hom(W,C_{\beta})\longrightarrow \Ext^2(\mathcal{I}_{{\mathcal C}_{\beta,W}},\mathcal{I}_{{\mathcal C}_{\beta,W}})\longrightarrow 0.
  \end{align}
Note that $T_{[\beta]}N\cong \Ext^1(M_{\beta},M_{\beta})$ by Corollary \ref{cor:tn}. Away from $\Gamma_X$, one has $R_{\beta}=W$ and $C_{\beta}=0$. Hence $\dim \Ext^1(\mathcal{I}_{{\mathcal C}_{\beta,W}},\mathcal{I}_{{\mathcal C}_{\beta,W}})=6$. On $\Gamma_X$, the surjectivity of the map $\nu_b$ in \eqref{numa} of Proposition
\ref{thm;blowup-center} implies that the map $\delta$ in \eqref{del} is also surjective. It follows again that $\dim \Ext^1\bigl(\mathcal{I}_{{\mathcal C}_{\beta,W}},
\mathcal{I}_{{\mathcal C}_{\beta,W}}\bigr)=6$. Therefore, by the same reasoning as in Proposition \ref{thm-deg4-iso}, we conclude that $ \Psi_6$ is an isomorphism.
\end{proof}

\begin{proposition}\label{exi6}
A general point of $G_6$ parametrizes a smooth rational sextic curve on $X$.
\end{proposition}
\begin{proof}
Let $C\subset X$ be a general free smooth rational sextic. As in the proof of Proposition \ref{exi5}, it remains to verify that the three-arrow Kronecker representation
\[
\beta\colon V_3\longrightarrow V_2\otimes B^\vee
\]
induced by the morphism $V_3\otimes K\longrightarrow V_2\otimes U$ arising from the Beilinson spectral sequence is stable. Any destabilizing subrepresentation would give rise to a proper curve of smaller degree contained in $C$, contradicting the integrality of the curve $C$.
\end{proof}

\subsection{Bisecant conics and Kronecker representations}\label{subsec-3conic}
In Corollary \ref{cor-blowup}, we proved that the blow-up of $N$ along $\Gamma_X\cong \Delta_X$ is isomorphic to the Hilbert scheme $\bH_6(X)$. This birational identification of $N$ with $\Hilb^3(\P(B))$ suggests the following interpretation. The base space $N$ parametrizes an unordered triple of conics in $X$ under the birational identification of $N$ with the Hilbert scheme of three points in $\PP(B)\cong \bH_2(X)$. In this subsection, we prove that for a general rational sextic curve $C$ in $X$, the image $\pi_6([C])\in N$ under the blow-up map $\pi_6$ in \eqref{blmap} parametrizes three conics which are bisecant to the given curve $C$.
Note that $B^\vee=\Hom(K, U)\cong\Hom(\wedge^2 U, K^\vee(-1))$.
\begin{lemma}\label{lm-res}
Let $\beta\colon V_3\lr V_2\otimes B^\vee$ be a general stable quiver representation. For general $\beta$, the fiber $\pi_6^{-1}[\beta]$ consists of a single point. Let $C_{\beta}\subset X$ be the sextic curve represented by this point. Then $I_{C_\beta}$ has a resolution:
\[
0\lr V_3\otimes\wedge^2U\stackrel{\delta_\beta}{\lr}V_2\otimes K^\vee(-1) \lr I_{C_\beta}\lr 0,
\]
where $\delta_\beta=(\mathrm{id}_{V_2}\otimes\delta)\circ(\beta\otimes\mathrm{id}_{\wedge^2U})$ and $
\delta\colon
B^\vee\otimes\wedge^2U
\longrightarrow
K^\vee(-1)$ is the evaluation morphism.
\end{lemma}
\begin{proof}
Combining the bundle map $\beta_X$ in Proposition \ref{thm;degree-six-surj} with the exact sequence \eqref{eq-af-one} and its dual sequence gives a commutative diagram
\begin{equation}\label{anres}
\begin{array}{ccccccccc}
0
&\longrightarrow&
V_3\otimes\wedge^2U
&\longrightarrow&
V_3\otimes A\otimes E
&\longrightarrow&
V_3\otimes K
&\longrightarrow&
0
\\
&&
\big\downarrow{\delta_\beta}
&&
\big\downarrow{L_\beta\otimes\mathrm{id}_E}
&&
\big\downarrow{\beta_X}
\\
0
&\longrightarrow&
V_2\otimes K^\vee(-1)
&\longrightarrow&
V_2\otimes A^\vee\otimes E
&\longrightarrow&
V_2\otimes U
&\longrightarrow&
0,
\end{array}
\end{equation}
where the map $L_\beta$ in the middle column is exactly the one in \eqref{kcmap}.
For general $\beta$, we have $\Ext^1(E,M_\beta)=0$, and the diagram \eqref{anres} induces an exact sequence
\[
\ses{R_\beta\otimes E}{M_\beta}{\mathrm{coker}({\delta_\beta})}.
\]
Comparing this sequence with \eqref{eq-I} gives $\mathrm{coker}(\delta_{\beta})\cong I_C$.
\end{proof}

The length-three subscheme $Z_{\beta}\subset \P(B)$ determines the three bisecant conics of a general sextic curve.

\begin{proposition}\label{prop-sextic-bisecants}
For a general stable Kronecker representation $\beta\colon V_3\lr V_2\otimes B^\vee$, the associated sextic curve $C=\pi_6^{-1}([\beta])$ has three bisecant conics parametrized by the point $[\beta]\in N$. These are all the bisecant conics to $C_{\beta}$.
\end{proposition}
\begin{proof}
The representation $\beta$ induces a bundle map on the projective plane $\PP(B)$:
\[
\tilde{\beta}\colon V_3\otimes \cO_{\PP(B)}(-1)\lr V_2\otimes \cO_{\PP(B)}.
\]
Let $Z_\beta\coloneqq D_{\leq 1}(\tilde{\beta})$ be the degeneracy locus of the map $\tilde{\beta}$. For general $\beta$, it is well known that there exists an exact sequence
\[
\ses{V_2^\vee\otimes \cO_{\PP(B)}(-3)}{V_3^\vee\otimes \cO_{\PP(B)}(-2)}{I_{Z_{\beta}}}
\]
such that $\mathrm{length}(Z_\beta)=3$. That is, the map $\beta$ provides a point $[Z_\beta]\in \mathrm{Hilb}^3(\PP(B))$ in the Hilbert scheme of three points in $\PP(B)$. Let us consider the case of three reduced points: $Z_\beta=\{z_1, z_2, z_3\}$. For $z=z_i=[\lambda]\in \PP(B)$, we can regard $\lambda\colon B^\vee\lr \CC$. That is, $S_z=\ker(\lambda)\subset B^\vee$, $\dim S_z=2$. Since $Z_\beta$ is the degeneracy locus of the bundle map $\tilde{\beta}$, the composition map
\[
\beta_z\colon V_3\longrightarrow V_2\otimes B^\vee \stackrel{\mathrm{id}_{V_2}\otimes \lambda}{\longrightarrow} V_2\otimes \CC \cong V_2
\]
has rank one. Define the canonical map $\pi_z\colon V_2 \lr V_2/ \mathrm{im}(\beta_z)=\overline{V}_2$. Since $(\pi_z\otimes\lambda) \beta =\pi_z \beta_z=0$, the map $(\pi_z\otimes \mathrm{id}_{B^\vee})\beta\colon V_3\lr \overline{V}_2\otimes B^\vee$ has image contained in $\overline{V}_2\otimes S_z$ and therefore, factors through this subspace. Hence, Lemma \ref{lm-res} and \eqref{star3} give a commutative diagram:
\[
\begin{tikzcd}[column sep=large,row sep=large] V_3\otimes\wedge^2U \arrow[r,"\delta_\beta"] \arrow[d,"\beta_z\otimes1"'] & V_2\otimes K^\vee(-1) \arrow[r] \arrow[d,"\pi_z\otimes1"] & I_C \arrow[d,two heads,"\rho_z"] \arrow[r] &0 \\ \bar V_2\otimes S_z\otimes\wedge^2U \arrow[r,"1\otimes\delta_{S_z}"'] & \bar V_2\otimes K^\vee(-1) \arrow[r,"1\otimes q_{S_z}"'] & \bar V_2\otimes\omega_{Q_z} \arrow[r] &0 . \end{tikzcd}\]
Therefore, there is a surjective map
\[
\rho_z\coloneqq I_C\twoheadrightarrow \overline{V}_2\otimes\omega_{Q_z}\cong \omega_{Q_z}.
\]
By the same argument as in Proposition \ref{prop-ubconic}, one can conclude that $Q_z$ is a bisecant conic to the curve $C$. Applying this construction to the three points of $Z_{\beta}$ gives the three bisecant conics.
\end{proof}

\section{Applications}\label{sec-dt}

\subsection{A Torelli-type theorem from rational sextics}
In this subsection, we prove that the sextic Hilbert scheme $\bH_6(X)$ determines the prime Fano threefold $X$.

\begin{proof}[Proof of Corollary \ref{coro:torelli-H6}]
Let $\mc H_X\coloneqq \bH_6(X)$ and $\mc H_{X'}\coloneqq \bH_6(X')$. By Corollary \ref{cor-blowup}, there are blow-up morphisms
\[
\pi_X\colon \mc H_X\cong \mathrm{Bl}_{\Gamma_X}N\longrightarrow N,\quad \pi_{X'}\colon \mc H_{X'}\cong \mathrm{Bl}_{\Gamma_{X'}}N\longrightarrow N,
\]
where $N=\mathbf{K}_3(3,2)$. By \cite{Dre88}, $\Pic(N)=\ZZ H_N$, where $H_N$ is the ample generator and $K_N=-3H_N$. Set $L_X\coloneqq \pi_X^*H_N$ and $L_{X'}\coloneqq \pi_{X'}^*H_N$, and denote the exceptional divisors by $E_X$ and $E_{X'}$, respectively. Let $\Phi\colon \mc H_X\xrightarrow{\sim}\mc H_{X'}$ be an isomorphism. We first show that $\Phi$ identifies the two blow-up morphisms. For $D\in N^1(\mc H_X)_{\Q}$, define
\[
q_D(u,v)\coloneqq \int_{\mc H_X} u\cup v\cup D^3,\quad u,v\in H^3(\mc H_{X},\Q).
\]
Let $j\colon E_X\hookrightarrow \mc H_X$ and $p\colon E_X\to \Gamma_X$ be the natural inclusion and projection. The blow-up formula gives
\[
H^3(\mc H_X,\Q)=j_{\ast}p^{\ast}H^1(\Gamma_X,\Q).
\]
Write $D=aL_X+bE_X$ and put $\xi=c_1(\OO_{E_X}(1))$. Since $j^*E_X=-\xi$, $j^*L_X=p^*(H_N|_{\Gamma_X})$ and $p_{\ast}(\xi^4)=1$, the projection formula gives
\[
q_D(j_{\ast}p^{\ast}\alpha, j_{\ast}p^{\ast}\beta)=b^3\int_{\Gamma_X}\alpha\cup\beta.
\]
The intersection form on $H^1(\Gamma_X,\Q)$ is nondegenerate. Therefore,
\[
\{D\in N^1(\mc H_X)_{\Q}\mid q_D=0\}=\Q L_X.
\]
This set is preserved by the isomorphism $\Phi$. Since $L_X$ and $L_{X'}$ are primitive nef classes, we have $\Phi^*L_{X'}=L_X$. Also, the canonical divisor formula gives $\Phi^{\ast}E_{X'}=E_X$. The section ring of $L_X$ recovers the contraction:
\[
N\cong \Proj \bigoplus_{m\geq 0}H^0(\mc H_X,mL_X).
\]
Consequently, $\Phi$ descends to an automorphism $\overline{\Phi}\colon N\xrightarrow{\sim} N$ and identifies the centers $\overline{\Phi}(\Gamma_X)=\Gamma_{X'}$ as closed subschemes of $N$. By Proposition \ref{thm;blowup-center}, we have $\Gamma_X=j_X(\Delta_X)$ and $\Gamma_{X'}=j_{X'}(\Delta_{X'})$, and this construction is reversible. More precisely, the maximal-minor locus $\Gamma_X\subset N$ determines the corresponding map
\[
\mc L_X\colon \mc V_3\otimes A_X\longrightarrow \mc V_2\otimes A_{X}^\vee,
\]
and hence determines the composition map $\mu_X\colon B^\vee \otimes A_X\to A_X^\vee$ up to changes of basis in $A_X$ and $B$. Thus, the isomorphism of pairs $(N,\Gamma_X)\cong (N,\Gamma_{X'})$ gives the corresponding net of quadrics. Finally, by \cite[Theorem 1.1 and Corollary 1.2]{Sch01}, the net of quadrics determines the Fano threefold $X$ up to isomorphism. Therefore, we conclude that $X\cong X'$.
\end{proof}

\subsection{Donaldson--Thomas type invariants}\label{sec-dt1}
Let $X$ be a smooth Fano $3$-fold and $\beta \in H_2(X,\mathbb{Z})$ be an effective curve class. Let $\bM_\beta(X)(=\bM)$ denote the moduli space of one-dimensional stable sheaves $F$ on $X$ satisfying $[F]=\beta$ and $\chi(F)=1$. Let $\pi_{\bM}\colon \bM_{\beta}(X)\times X\lr \bM_{\beta}(X)$ and $\pi_X\colon \bM_{\beta}(X)\times X\lr X$ be the projections, and $\cF$ be a normalized universal sheaf. For $\gamma \in \rH^{4-2i}(X,\mathbb{Z})$, the \emph{insertion} is defined by
\[
\tau_{i}(\gamma)=\pi_{\bM*}\left(\pi_X^*\gamma\cup \{\mathrm{ch}(\cF)\cdot \text{td}(K_X)^{-1}\}_{i+2}\right).
\]
The corresponding invariant over $\bM_\beta(X)$ is
\begin{equation}\label{eq-dt-descendant}
\langle \tau_i(\gamma) \rangle_{\beta}\coloneqq \int_{[\bM_\beta(X)]^{\mathrm{vir}}} \tau_i (\gamma).  
\end{equation}
For $i=0$, the integral $\langle \tau_0(\gamma) \rangle_{\beta}$ is called the \emph{primary} Donaldson--Thomas invariant. For $i=1$, we call $\langle \tau_i(\gamma) \rangle_{\beta}$ a \emph{descendant invariant}. In this subsection, we compute the primary and descendant DT-invariants in degrees $4$, $5$, and $6$.
Since the moduli spaces of interest are smooth, the invariant in \eqref{eq-dt-descendant} can be written as
\[
\langle \tau_i(\gamma) \rangle_{\beta}=\int_{[\bM_\beta(X)]} \tau_i(\gamma) e(\mathrm{Ob}),
\]
where $e(\mathrm{Ob})$ is the Euler class of the obstruction bundle $\mathrm{Ob}$ with fiber $\mathrm{Ob}|_{[F]}=\Ext_X^2(F,F)$ on $\bM_\beta(X)$. Since $X$ is Fano, the K-class  of the obstruction bundle is $[\mathrm{Ob}]
=[T_{\bM_\beta(X)}]+[R\pi_{\bM *}R\mathcal Hom(\cF,\cF)]-[\mathcal O]$ .
For the computation, we recall the Chow ring of the Kronecker moduli space $N=\bK_3(3,2)$, which will be used in the degree $6$ computation. The computations below were verified using Macaulay2 (\cite{M2}).

Let $\mc V_3$ and $\mc V_2$ be the universal bundles of ranks $3$ and $2$
on $N$, respectively. Set $b_i\coloneqq c_i(\mc V_3^\vee)$ and $d_i\coloneqq c_i(\mc V_2^\vee)$. In particular, $b_1=d_1$.

\begin{lemma}\label{lem:chow-ring-N}
The Chow ring of $N$ is given by
\[
A^*(N)_{\Q}\cong \mathbb Q[b_1,b_2,d_2]/I_N,
\]
where
\[
I_N=\bigl(
b_1^2d_2-3d_2^2,
 b_1^2b_2-b_2d_2-3d_2^2,b_1^4+3b_2^2-9b_2d_2-3d_2^2,
 2b_1b_2d_2-3b_1d_2^2,3b_1b_2^2-7b_1d_2^2
\bigr).
\]
Moreover,
\[
b_3=\frac{-b_1^3+4b_1d_2}{3},
\qquad
\int_Nd_2^3=2.
\]
\end{lemma}

\begin{proof}
The presentation of $A^*(N)$ and the relations follow from \cite[Proposition 4.4]{CM15}. On the other hand, consider the natural embedding $j\colon \mathbb P^{2\vee}\hookrightarrow N$. The class of $\mathbb P^{2\vee}$ in $N$ and the restriction of $d_2$
are given by
\[
[\mathbb P^{2\vee}]
=
-3b_1^2b_2+5b_1^2d_2,
\quad
j^*(d_2)=3h^2,
\]
where $h$ is the hyperplane class of $\mathbb P^{2\vee}$.
Therefore,
\[
\int_N d_2^3=
\frac{2}{3}\int_N
(-3b_1^2b_2+5b_1^2d_2)d_2=
\frac{2}{3}\int_{\mathbb P^{2\vee}} j^*(d_2)=
\frac{2}{3}\int_{\mathbb P^{2\vee}}3h^2
=2,
\]
where the first equality follows from the relations in $A^*(N)$.
\end{proof}
Using the descriptions of the universal families for $4\leq d\leq6$, together with Grothendieck--Riemann--Roch,
we obtain the following.

\begin{proposition}[cf. \protect{\cite[Proposition 3.17]{CLW24}}]\label{v22dt}
For $4\leq d\leq 6$, the primary/descendant invariants $\langle \tau_i(h_{2-i})\rangle_{d}$, $i=0, 1$ are listed in Table \ref{tab:invariants_d1to4} for the cohomology classes $h_1=H$, $h_2=\frac{H^2}{22}$ of $X$.
\begin{table}[htbp] \begin{center} \begin{tabular}{|l||M{2.2cm}|M{2.2cm}|} \hline $d$ & $i=0$ & $i=1$\\ \hline\hline $4$ & $168$ & $-168$\\ \hline $5$ & $1176$ & $-2688$\\ \hline $6$ & $9030$ & $-26880$\\ \hline \end{tabular} \caption{The invariants $\langle \tau_i(h_{2-i})\rangle_{d}$ for $4\leq d\leq 6$} \label{tab:invariants_d1to4} \end{center} \end{table}
\end{proposition}
\begin{proof}
We only give the details for degree $d=6$ since the other cases are typical ones. Recall that
\[
\bH_6(X)\cong G_6\cong\mathrm{Bl}_{\Gamma_X}N,
\quad
N=\bK_3(3,2),\quad \pi_6\coloneqq G_6\longrightarrow N.
\] 
Let $\mathcal W_6$ denote the tautological
rank-four bundle on $G_6$. By the construction of the universal family in \eqref{eq:un6}, there is an
exact sequence
\[
0
\longrightarrow
\mathcal W_6\boxtimes E
\stackrel{\mathbf{u}_6}{\longrightarrow}
\mathcal M_6
\longrightarrow
I_{\mathcal C}\otimes\Lambda_6
\longrightarrow0,
\]
where $\Lambda_6 = (\det \mc V_3)^5 \otimes (\det \mc V_2)^{-3} \otimes (\det\mathcal W_6)^{-2}$ and $\mathcal M_6$ fits into the exact sequence
\[
0
\longrightarrow
\mathcal M_6
\longrightarrow
\mc V_3\boxtimes K
\longrightarrow
\mc V_2\boxtimes U
\longrightarrow0.
\]
Consequently, after normalizing the family by $\Lambda_6^{-1}$, the universal structure sheaf $\mathcal O_{\mathcal C}$ has the following class in
$K_0(G_6\times X)$:
\begin{equation}\label{eq:no}
[\mathcal O_{\mathcal C}]
=
[\mathcal O]
-
[\Lambda_6^{-1}\mc V_3\boxtimes K]
+
[\Lambda_6^{-1}\mc V_2\boxtimes U]
+
[\Lambda_6^{-1}\mathcal W_6\boxtimes E].
\end{equation}

Put $v_3\coloneqq c_1(\mc V_3)$, $v_2\coloneqq c_1(\mc V_2)$ and $w\coloneqq c_1(\mathcal W_6)$.
Then $\lambda\coloneqq c_1(\Lambda_6)=5v_3-3v_2-2w$. Expanding \eqref{eq:no} to first order in the
cohomology of $G_6$, we obtain
\[
[\ch(\mathcal O_{\mathcal C})]_{\deg_{G_6}=1}
=
\lambda
\bigl(
3\ch(K)-2\ch(U)-4\ch(E)
\bigr)\\
-v_3\ch(K)
+v_2\ch(U)
+w\ch(E).
\]
Multiplying this expression by
$\mathrm{td}(K_X)^{-1}$ and taking the relevant components gives
\[
\begin{aligned}
\tau_0(h_2)
=
2v_3-v_2-w,\quad
\tau_1(h_1)
=
-12v_3+8v_2+5w.
\end{aligned}
\]
We now use the notation of Lemma \ref{lem:chow-ring-N}.
Set $s\coloneqq c_1(\mathcal W_6^\vee)$. Since $c_1(\mc V_3^\vee)=c_1(\mc V_2^\vee)=b_1$, we have $v_3=v_2=-b_1$ and $w=-s$. Hence, the descendant insertions take the particularly simple form
\begin{equation}\label{eq:sextic-insertions}
\begin{aligned}
\tau_0(h_2)=s-b_1,\quad
\tau_1(h_1)=4b_1-5s.
\end{aligned}
\end{equation}

We next describe the obstruction bundle explicitly. Let $\pi\colon G_6\times X\to G_6$ and $\mathcal R_6\coloneqq R\pi_*R\mathcal Hom(\mathcal O_{\mathcal C},\mathcal O_{\mathcal C})$. By \eqref{eq:no} and the Euler pairings of the exceptional bundles of $X$, we obtain
\[
\begin{aligned}
[\mathcal R_6]
={}&
[\mathcal O]
+[\cE nd(\mc V_3)]
+[\cE nd(\mc V_2)]
+[\cE nd(\mathcal W_6)]\\
&-14[\Lambda_6\otimes \mc V_3^\vee]
-3[\mc V_3^\vee\otimes \mc V_2]
+7[\Lambda_6\otimes \mc V_2^\vee]\\
&+8[\Lambda_6\otimes\mathcal W_6^\vee]
-4[\mathcal W_6^\vee\otimes \mc V_3]
+4[\mathcal W_6^\vee\otimes \mc V_2].
\end{aligned}
\]
On the other hand, from the construction of $G_6$ in \eqref{eq:relze}, we have
\[
\begin{aligned}
[TG_6]
={}&
3[\mc V_3^\vee\otimes \mc V_2]
-[\cE nd(\mc V_3)]
-[\cE nd(\mc V_2)]
+[\mathcal O]\\
&+4[\mathcal W_6^\vee\otimes \mc V_3]
-[\cE nd(\mathcal W_6)]
-4[\mathcal W_6^\vee\otimes \mc V_2].
\end{aligned}
\]
Therefore
\[\begin{aligned}
[\mathrm{Ob}_6]
&= [TG_6]+[\mathcal R_6]-[\mathcal O] \\
&= [\mathcal O]
-14[\Lambda_6\otimes \mc V_3^\vee]
+7[\Lambda_6\otimes \mc V_2^\vee]
+8[\Lambda_6\otimes \mathcal W_6^\vee].
\end{aligned}
\]
By the relations in
Lemma \ref{lem:chow-ring-N} and \eqref{eq:sextic-insertions}, we obtain
\[
\pi_{*}
\left(
e(\mathrm{Ob}_6)\tau_0(h_2)
\right)
=
4515\,d_2^3,
\quad
\pi_{*}
\left(
e(\mathrm{Ob}_6)\tau_1(h_1)
\right)
=
-13440\,d_2^3.
\]
Since $\int_Nd_2^3=2$ (Lemma \ref{lem:chow-ring-N}),
we obtain the numbers in the last row of Table \ref{tab:invariants_d1to4}.
\end{proof}

\bibliographystyle{alpha}
\bibliography{Library.bib}

\end{document}